\documentclass[11pt]{amsart}
\usepackage[top=2.5cm,bottom=2.5cm,left=2.5cm,right=2.5cm]{geometry}
\usepackage[english]{babel}
\usepackage{times}
\usepackage{amssymb,amsthm}
\newcommand{\Pend}{\mathsf P}  % P(X) = alpha X + beta J_0 X
\newcommand{\Qend}{\mathsf Q}  % Q(X) = gamma X + delta J_0 X

\title[Symplectic and almost contact homogeneous structures]
{Infinitesimally homogeneous manifolds with symplectic and almost contact structure groups}
\author{Carlos Alberto Mar\'in Arango}
\address{
Instituto de Matemáticas y Estadística,
Facultad de Ciencias Exactas y Naturales,
Universidad de Antioquia,
Medellín, Colombia
}
\email{calberto.marin@udea.edu.co}

\author{Alejandro Pinilla Barrera}
\address{
Instituto de Matemáticas y Estadística,
Facultad de Ciencias Exactas y Naturales,
Universidad de Antioquia,
Medellín, Colombia
}
\email{alejandro.pinilla@udea.edu.co}

\begin{document}
\numberwithin{equation}{section}
\theoremstyle{plain}\newtheorem{teo}{Theorem}[section]
\theoremstyle{plain}\newtheorem{prop}[teo]{Proposition}
\theoremstyle{plain}\newtheorem{lema}[teo]{Lemma}
\theoremstyle{definition}\newtheorem{defn}[teo]{Definition}
\theoremstyle{remark}\newtheorem{remark}[teo]{Remark}
\theoremstyle{plain}\newtheorem{cor}[teo]{Corollary}
\theoremstyle{definition}\newtheorem{example}[teo]{Example}

\begin{abstract}
We study infinitesimally homogeneous affine manifolds with prescribed
structure group, that is, triples \((M,\nabla,P)\) whose torsion,
curvature and inner torsion have constant components in the frames of a
\(G\)-structure \(P\). For the real symplectic group
\(\operatorname{Sp}(2n,\mathbb R)\) we prove that these are precisely the
flat torsion-free symplectic affine structures. For
\(G=\operatorname{U}(n)\times\{1\}\), whose \(G\)-structures are almost
contact metric structures, we classify the characteristic tensors by
first treating the torsion-free case and then applying invariant
torsion deformations. We show that the underlying almost contact metric
structure is trans-Sasakian of type \((\alpha_T,\beta_T)\), with
\(\alpha_T=\delta-\tau_3/2\) and \(\beta_T=\alpha-a-\tau_1\), where
\(\tau_1,\tau_3\) are torsion parameters and \(a,\alpha,\delta\)
inner-torsion parameters. The algebraic relation among these parameters
becomes Marrero's identity \(\alpha_T\beta_T=0\), and the horizontal
distribution is contact exactly when \(\alpha_T\neq0\). Locally, the
manifold is isomorphic to a member of one of the three model families
appearing in Tanno's classification of maximally symmetric almost
contact metric manifolds: a complex space form times a line, an
\(\alpha_T\)-Sasakian space form, or real hyperbolic space with its
Kenmotsu structure. Finally, the infinitesimally homogeneous connections
on these model \(G\)-structures are exactly those of the form
\(\nabla^g+S\), where \(\nabla^g\) is the Levi-Civita connection and
\(S\) is induced by an element of the seven-dimensional space of
\(G\)-equivariant bilinear maps.
\end{abstract}

\maketitle

\section{Introduction}
\label{sec:introduction}

Infinitesimally homogeneous manifolds with $G$-structure provide a
natural framework for studying affine manifolds whose local geometry is
encoded by constant algebraic data. More precisely, if
\(G\subset \operatorname{GL}(\mathbb R^n)\) is a Lie subgroup, an affine
manifold with \(G\)-structure is a triple \((M,\nabla,P)\), where \(M\)
is an \(n\)-dimensional smooth manifold, \(\nabla\) is an affine
connection on \(M\), and \(P\subset \mathrm{FR}(TM)\) is a
\(G\)-structure. Such a triple is called infinitesimally homogeneous
when the components of the curvature tensor, the torsion tensor and
the inner torsion of the \(G\)-structure are constant when expressed
with respect to any frame belonging to \(P\). These constant
representatives are called the characteristic tensors of the structure
and will be denoted by \(R_0,\, T_0,\,\mathfrak I_0.\)

The term goes back to Singer~\cite{Singer}, for whom a Riemannian
manifold is infinitesimally homogeneous when, up to a sufficiently high
finite order, its curvature tensor and its covariant derivatives have
the same components at all points in suitable orthonormal frames; he
proved that such manifolds are locally homogeneous. See also
Ambrose--Singer~\cite{AmbroseSinger} and
Tricerri--Vanhecke~\cite{TricerriVanhecke}.
Opozda~\cite{Opozda97,Opozda98} developed analogous results for affine
connections and for $G$-structures.

The homogeneity condition considered here is of order zero relative
to the $G$-structure: no covariant derivatives of the torsion or
curvature are prescribed; instead, the torsion, curvature and inner
torsion are required to have constant components in the frames of $P$.
This notion appears in the work of Piccione and
Tausk~\cite{piccione-tausk} on $G$-structure preserving affine
immersions, where infinitesimal homogeneity allows the Gauss, Codazzi,
Ricci and torsion equations to be expressed intrinsically in terms of
the characteristic tensors of the target. When $G$ is compact, as for
$\operatorname{U}(n)\times\{1\}$, this condition implies local
homogeneity (Proposition~\ref{prop:local-homogeneity}); compare
Opozda~\cite{Opozda98}. The advantage of this formulation for the
present purposes is that the classification reduces to an algebraic
problem for the characteristic tensors.

For a fixed structure group $G$, the tensors \((T_0,R_0,\mathfrak I_0)\) are not
arbitrary. They must be invariant under the natural action of $G$ and
satisfy algebraic compatibility relations together with Bianchi-type
identities. In~\cite{marin}, necessary and sufficient algebraic
conditions for such a triple to arise as the characteristic data of an
infinitesimally homogeneous manifold were established. An important
feature of that classification framework is that the problem with
arbitrary torsion can be reduced to the torsion-free case.

More precisely, if \((M,\nabla, P)\) is an infinitesimally homogeneous
manifold with characteristic torsion tensor \(T_0\), then the
symmetrized connection \(\nabla^s=\nabla-\frac12T\) is torsion-free and
\((M,\nabla^s, P)\) is again infinitesimally homogeneous. Conversely,
starting from a torsion-free infinitesimally homogeneous manifold and a
\(G\)-equivariant skew-symmetric map
\(t_0\in\operatorname{Hom}_G(\Lambda^2 \mathbb{R}^n,\mathbb{R}^n)\), one
recovers an infinitesimally homogeneous manifold with torsion by setting
\(\nabla=\nabla^s+\frac12t\), where \(t\) denotes the tensor field induced by \(t_0\) by using the frames in \(P\). Thus, for a fixed structure group, the
classification problem naturally splits into two parts: the
classification of the torsion-free characteristic data and the
determination of the invariant tensors that may occur as torsion.

Following this program, several classical structure groups have already
been studied. In~\cite{Marin-Blazquez}, explicit classifications were
obtained for groups including the identity group, finite groups, the
diagonal group, the special linear group, the orthogonal group, and the
unitary group. The purpose of the present paper is to continue this
classification program for the two groups
\(\operatorname{Sp}(2n,\mathbb R)\) and
\(\operatorname{U}(n)\times\{1\}\subset\operatorname{GL}(2n+1,\mathbb R)\).

For the real symplectic group, the classification is rigid: all the
characteristic tensors necessarily vanish. Consequently,
infinitesimally homogeneous structures with structure group
\(\operatorname{Sp}(2n,\mathbb R)\) are precisely flat torsion-free
symplectic affine structures. This answers the question about the
symplectic group raised in~\cite{Marin-Blazquez}, and shows that a
non-flat torsion-free symplectic connection can be infinitesimally
homogeneous only with respect to a reduction of the symplectic frame
bundle to a proper subgroup of \(\operatorname{Sp}(2n,\mathbb R)\)
(Remark~\ref{rem:special-symplectic}). The case \(G=\operatorname{U}(n)\times\{1\}\) is considerably richer. A
\(G\)-structure of this type is naturally equivalent to almost contact
metric data on a \((2n+1)\)-dimensional manifold. We determine the
invariant candidates for the characteristic tensors, solve the
torsion-free classification problem, and then obtain the general case
with nonzero torsion. In terms of the torsion parameters
\(\tau_1,\tau_2,\tau_3\) and the inner torsion parameters
\(a,b,\alpha,\beta,\gamma,\delta\), the admissible torsion and inner-torsion data are
characterized by the relation
\((2\delta-\tau_3)(\alpha-a-\tau_1)=0\) (Corollary~\ref{cor:parameter-classification}).

These parameters have a transparent geometric meaning. The underlying
almost contact metric structure is trans-Sasakian of constant type
\((\alpha_T,\beta_T)\), with
\[
  \alpha_T=\delta-\frac{\tau_3}{2},
  \qquad
  \beta_T=\alpha-a-\tau_1,
\]
and the relation among the torsion and inner-torsion parameters
becomes Marrero's identity $\alpha_T\beta_T=0$~\cite{marrero}.
Accordingly, the underlying almost contact metric manifold is locally
isomorphic to one of three models: a product
$N^{2n}(\kappa)\times\mathbb R$, with $N^{2n}(\kappa)$ a complex space
form of constant holomorphic sectional curvature $\kappa$, when
$\alpha_T=\beta_T=0$; an $\alpha_T$-Sasakian space form, of constant
$\phi$-sectional curvature $\kappa-3\alpha_T^2$, when $\alpha_T\neq0$;
or the real hyperbolic space $\mathbb H^{2n+1}(-\beta_T^2)$ with its
Kenmotsu structure when $\beta_T\neq0$. In particular, the horizontal
distribution is contact exactly in the $\alpha_T$-Sasakian case. Moreover, the
infinitesimally homogeneous connections on these models are exactly
those of the form $\nabla=\nabla^g+S$, where $\nabla^g$ is the
Levi-Civita connection and $S$ is induced by an element of the
seven-dimensional space
$\operatorname{Hom}_G(\mathbb R^{2n+1}\otimes\mathbb R^{2n+1},
\mathbb R^{2n+1})$ of $G$-equivariant bilinear maps
(Theorem~\ref{thm:structure}).

These are, up to local isomorphism and homothety, the model families
appearing in Tanno's classification~\cite{Tanno69} of almost contact
metric manifolds whose automorphism group has the maximal dimension
\((n+1)^2\). Our contribution is to describe all the affine
connections that are infinitesimally homogeneous with respect to them,
and to read the geometry directly from the characteristic parameters.
In particular, the relation \(\alpha_T\beta_T=0\) is obtained uniformly
for all \(n\geq1\), including dimension three, where Marrero's structure
theorem does not apply. Moreover, the torsion-free classification is
purely algebraic: the second Bianchi identity is settled by realizing
the relevant data through left-invariant connections on Lie groups
(Lemma~\ref{lem:lie-realization}).
In the Chinea--Gonz\'alez classification~\cite{ChineaGonzalez}, the
underlying structures lie in the class \(\mathcal C_5\oplus\mathcal C_6\).

The paper is organized as follows. Section~\ref{sec:preliminaries}
recalls the basic theory of infinitesimally homogeneous manifolds, the
algebraic conditions satisfied by their characteristic tensors, the
deformation procedure, some representation-theoretic facts, and the
local equivalence and local homogeneity of these structures.
Section~\ref{sec:symplectic} contains the classification for
\(\operatorname{Sp}(2n,\mathbb R)\). Section~\ref{sec:unitary-one} is
devoted to \(\operatorname{U}(n)\times\{1\}\): we determine the invariant
candidates for the characteristic tensors, classify the torsion-free
stratum, and obtain the general torsional case by invariant
deformation. Subsection~\ref{subsec:underlying} identifies the
underlying almost contact metric manifolds through trans-Sasakian
geometry, proves the structure theorem, and derives from it the
consequences for contact structures and compatible connections. We end with a brief remark on the
use of these models as ambient spaces for the immersion problems
of~\cite{piccione-tausk}.

%%%%%%%%%%%%%%%%%%%%%%%%%%%%%%%%%
%%%%%%%%%%%%%%%%%%%%%%%%%%%%%%%%%

\section{Notation and preliminaries}
\label{sec:preliminaries}

\subsection{Infinitesimally homogeneous manifolds}
\label{subsec:inf-hom}

Let \(G\subseteq \operatorname{GL}(\mathbb{R}^n)\) be a Lie subgroup
with Lie algebra \(\mathfrak g\). A \(G\)-structure on a smooth manifold
\(M\) with \(\dim M=n\) is a principal \(G\)-subbundle
\(P\subseteq \operatorname{FR}(TM)\) of the frame bundle of \(M\). Thus,
for each \(x\in M\), an element \(p\in P_x\) is a linear isomorphism
\(p:\mathbb{R}^n\to T_xM\), and two frames in \(P_x\) differ by right
composition with an element of \(G\). Let \(\nabla\) be an affine
connection on \(M\). We shall call the triple \((M,\nabla,P)\) an affine
manifold with \(G\)-structure. The connection \(\nabla\) determines its
torsion tensor \(T\) and curvature tensor \(R\), which we normalize by
\[
T(X,Y)=\nabla_XY-\nabla_YX-[X,Y],
\quad
R(X,Y)=\nabla_X\nabla_Y-\nabla_Y\nabla_X-\nabla_{[X,Y]}.
\]
In addition, the \(G\)-structure \(P\) has an associated tensor, called
the inner torsion, which measures the failure of \(\nabla\) to be
compatible with \(P\). We briefly recall its construction.

Let \(\omega^\nabla\) denote the connection form of \(\nabla\) on
\(\operatorname{FR}(TM)\), and let
\(\pi:\mathfrak{gl}(\mathbb{R}^n)\to\mathfrak{gl}(\mathbb{R}^n)/\mathfrak g\)
be the quotient map. If \(s:U\to P\) is a local section, we set
\[
\Lambda^s_x(u)=(s^*\omega^\nabla)_x\bigl(s(x)u\bigr)
\in\mathfrak{gl}(\mathbb R^n),
\quad x\in U,\ u\in\mathbb R^n.
\]
With this convention,
\begin{equation}
\label{eq:lift-convention}
\nabla_{s(x)u}(s\,w)=s(x)\bigl(\Lambda^s_x(u)w\bigr),
\quad u,w\in\mathbb R^n,
\end{equation}
where \(s\,w\) denotes the local vector field \(x\mapsto s(x)w\).
In particular, \(\nabla\) preserves \(P\) if and only if \(\Lambda^s\)
takes values in \(\mathfrak g\); the obstruction to compatibility is
therefore the class of \(\Lambda^s_x(u)\) modulo \(\mathfrak g\).
Passing to the quotient gives a linear map
\[
\mathfrak I^P_x:T_xM\to\mathfrak{gl}(T_xM)/\mathfrak g_x,
\quad
\mathfrak I^P_x\bigl(s(x)u\bigr)
=\overline{\operatorname{Ad}}_{s(x)}\bigl(\pi(\Lambda^s_x(u))\bigr),
\]
where \(\mathfrak g_x=\operatorname{Ad}_p(\mathfrak g)\) for any
\(p\in P_x\), and
\(\overline{\operatorname{Ad}}_p:\mathfrak{gl}(\mathbb{R}^{n})/\mathfrak g
\to\mathfrak{gl}(T_xM)/\mathfrak g_x\) denotes the map induced by
conjugation with \(p\). This map is independent of the choice of local
section and is called the inner torsion of \(P\) with respect to
\(\nabla\). In particular, \(\mathfrak I^P=0\) if and only if \(\nabla\)
is compatible with the \(G\)-structure.

\begin{defn}
\label{def:inf-hom}
The triple \((M,\nabla,P)\) is said to be \emph{infinitesimally
homogeneous} if the tensors \(T\), \(R\), and \(\mathfrak I^P\) have
constant representations in the frames of \(P\). Equivalently, there
exist linear maps
\[
T_0:\Lambda^2\mathbb{R}^{n}\to \mathbb{R}^{n},
\quad
R_0:\Lambda^2\mathbb{R}^{n}\to\mathfrak{gl}(\mathbb{R}^{n}),
\quad
\mathfrak I_0:\mathbb{R}^{n}\to \mathfrak{gl}(\mathbb{R}^{n})/\mathfrak g
\]
such that, for every \(x\in M\), \(p\in P_x\), and
\(u,v,w\in \mathbb{R}^{n}\),
\[
T_x(pu,pv)=p\,T_0(u,v),
\quad
R_x(pu,pv)pw=p\,R_0(u,v)w,
\quad
\mathfrak I^P_x(pu)
=
\overline{\operatorname{Ad}}_p\bigl(\mathfrak I_0(u)\bigr).
\]
The tensors \(T_0\), \(R_0\), and \(\mathfrak I_0\) are called the
\emph{characteristic tensors} of \((M,\nabla,P)\).
\end{defn}

The characteristic tensors are necessarily invariant under the natural
action of \(G\). More precisely, for every \(g\in G\) and
\(u, v \in \mathbb{R}^{n}\),
\[
T_0(gu,gv)=gT_0(u,v),
\quad
R_0(gu,gv)=\operatorname{Ad}_g\bigl(R_0(u,v)\bigr),
\quad
\mathfrak I_0(gu)=\overline{\operatorname{Ad}}_g\bigl(\mathfrak I_0(u)\bigr).
\]
For the inner torsion, this invariance condition is expressed in the
quotient: if \(\lambda:\mathbb R^n\to \mathfrak{gl}(\mathbb{R}^n)\) is a
lift of \(\mathfrak I_0\), the \(G\)-equivariance of \(\mathfrak I_0\) is
equivalent to
\[
\operatorname{Ad}_g\bigl(\lambda(u)\bigr)-\lambda(gu)\in\mathfrak g,
\quad g\in G,\ u\in \mathbb{R}^n.
\]
Moreover, the condition on the inner torsion in Definition~\ref{def:inf-hom} means that \(\Lambda^s_x(u)-\lambda(u)\in\mathfrak g\) for every local section \(s\) of \(P\), every \(x\) in its domain, and every \(u\in\mathbb R^n\).

For \(L\in\mathfrak{gl}(\mathbb R^n)\) and a multilinear map \(F\) on
\(\mathbb R^n\) with values in \(\mathbb R^n\) or in
\(\mathfrak{gl}(\mathbb R^n)\), let \(D_LF\) denote the derivation
induced by \(L\),
\[
(D_LF)(u_1,\dots,u_k)=L\cdot F(u_1,\dots,u_k)
-\sum_{i=1}^kF(u_1,\dots,Lu_i,\dots,u_k),
\]
where \(L\cdot w=Lw\) for \(w\in\mathbb R^n\) and \(L\cdot A=[L,A]\) for
\(A\in\mathfrak{gl}(\mathbb R^n)\). Differentiating the equivariance
relations for \(R_0\) and \(\mathfrak I_0\) at the identity gives, for
every \(L\in\mathfrak g\),
\begin{equation}
\label{eq:inf-curvature-invariance}
D_LR_0=0,
\end{equation}
\begin{equation}
\label{eq:inf-inner-invariance}
(D_L\lambda)(u)\in\mathfrak g,\quad u\in\mathbb R^n.
\end{equation}

\subsection{Algebraic conditions in the torsion-free case}
\label{subsec:algebraic-conditions}

The characteristic tensors of an infinitesimally homogeneous triple are
not arbitrary invariant tensors. They must also satisfy the algebraic
conditions of \cite[Theorem~1]{marin}, which are the constant-component
counterparts of the structural equations of an affine connection with
\(G\)-structure.

Assume \(T_0=0\), and let
\(\lambda:\mathbb R^n\to \mathfrak{gl}(\mathbb{R}^n)\) be any lift of
\(\mathfrak I_0\). The algebraic compatibility relation between the
curvature and the inner torsion is
\begin{equation}
\label{eq:torsion-free-compatibility}
R_0(u,v)-[\lambda(u),\lambda(v)]
+\lambda\bigl(\lambda(u)v-\lambda(v)u\bigr)
\in\mathfrak g,
\quad u,v\in \mathbb{R}^{n}.
\end{equation}
Since \(T_0=0\), the curvature satisfies the first Bianchi identity
\begin{equation}
\label{eq:first-bianchi}
\mathfrak S_{u,v,w}R_0(u,v)w=0,
\quad u,v,w \in \mathbb{R}^{n},
\end{equation}
where \(\mathfrak S_{u,v,w}\) denotes cyclic summation in \(u,v,w\). In
terms of the lift \(\lambda\), the second Bianchi identity becomes
\begin{equation}
\label{eq:second-bianchi}
\mathfrak S_{u,v,w}\bigl(D_{\lambda(u)}R_0\bigr)(v,w)=0.
\end{equation}
Both \eqref{eq:torsion-free-compatibility} and
\eqref{eq:second-bianchi} are independent of the chosen lift: replacing
\(\lambda\) by \(\lambda+\theta\), with \(\theta\) taking values in
\(\mathfrak g\), changes the left-hand side of
\eqref{eq:torsion-free-compatibility} by an element of \(\mathfrak g\),
by \eqref{eq:inf-inner-invariance}, and changes that of
\eqref{eq:second-bianchi} by
\(\mathfrak S_{u,v,w}(D_{\theta(u)}R_0)(v,w)\), which vanishes by
\eqref{eq:inf-curvature-invariance}.

\begin{teo}[{\cite[Theorems~1 and~2]{marin}}]
\label{thm:algebraic-characterization}
Let \(G\subseteq\operatorname{GL}(\mathbb R^n)\) be a Lie subgroup, and
let \(R_0:\Lambda^2\mathbb R^n\to\mathfrak{gl}(\mathbb R^n)\) and
\(\mathfrak I_0:\mathbb R^n\to\mathfrak{gl}(\mathbb R^n)/\mathfrak g\) be
\(G\)-equivariant linear maps. Then \((R_0,\mathfrak I_0)\) are the
curvature and inner torsion characteristic tensors of a torsion-free
infinitesimally homogeneous affine manifold with \(G\)-structure if and
only if, for some (equivalently, for any) lift \(\lambda\) of
\(\mathfrak I_0\), the relations
\eqref{eq:torsion-free-compatibility}, \eqref{eq:first-bianchi}, and
\eqref{eq:second-bianchi} hold.
\end{teo}

Thus, in the torsion-free case, the algebraic classification problem is
reduced to determining all \(G\)-equivariant pairs
\((R_0,\mathfrak I_0)\) satisfying
\eqref{eq:torsion-free-compatibility}, \eqref{eq:first-bianchi}, and
\eqref{eq:second-bianchi}.

\subsection{The torsion-free stratum and torsion deformations}
\label{subsec:deformations}

Let \((M,\nabla,P)\) be an infinitesimally homogeneous affine
manifold with \(G\)-structure and characteristic tensors
\(T_0,R_0,\mathfrak I_0\). A \(G\)-equivariant bilinear map
\(S_0:\mathbb R^n\times\mathbb R^n\to\mathbb R^n\) induces a
well-defined tensor field
\(S\in\Gamma(T^*M\otimes T^*M\otimes TM)\) whose components
in every frame of \(P\) are \(S_0\). Write
\(S_{0,u}:=S_0(u,\cdot)\).

Equivariance implies \(D_LS_0=0\) for every \(L\in\mathfrak g\),
where \(D_L\) is the derivation defined in
Subsection~\ref{subsec:inf-hom}. Thus, for any lift \(\lambda\)
of \(\mathfrak I_0\), the expression \(D_{\lambda(u)}S_0\)
depends only on the class
\(\lambda(u)+\mathfrak g=\mathfrak I_0(u)\). Moreover, since
\(S\) has constant components in the frames of \(P\) and
\(\Lambda^s_x(u)-\lambda(u)\in\mathfrak g\) for every local section
\(s\) of \(P\), \eqref{eq:lift-convention} gives
\[
(\nabla_{pu}S)(pv,pw)
=p\,\bigl(D_{\lambda(u)}S_0\bigr)(v,w),
\qquad p\in P_x.
\]
For \(\nabla'=\nabla+S\), \eqref{eq:lift-convention} also gives
\(\Lambda'^{\,s}_x(u)=\Lambda^s_x(u)+S_{0,u}\), so that
\(\lambda'(u):=\lambda(u)+S_{0,u}\) satisfies
\(\Lambda'^{\,s}_x(u)-\lambda'(u)=\Lambda^s_x(u)-\lambda(u)\in\mathfrak g\).
The standard formulas for the torsion and curvature of
\(\nabla'\) therefore show that \((M,\nabla',P)\) is
infinitesimally homogeneous, with characteristic tensors
\begin{equation}
\label{eq:S-deformation-TI}
T_0'(u,v)=T_0(u,v)+S_0(u,v)-S_0(v,u),
\qquad
\mathfrak I_0'(u)=\lambda'(u)+\mathfrak g,
\end{equation}
and
\begin{equation}
\label{eq:S-deformation-R}
\begin{aligned}
R_0'(u,v)w={}&R_0(u,v)w
+\bigl(D_{\lambda(u)}S_0\bigr)(v,w)
-\bigl(D_{\lambda(v)}S_0\bigr)(u,w)\\
&+S_0\bigl(T_0(u,v),w\bigr)
+[S_{0,u},S_{0,v}]w.
\end{aligned}
\end{equation}

In particular, taking \(S_0=\tfrac12A_0\), where \(A_0\) is
\(G\)-equivariant and skew-symmetric, recovers the torsion
deformation lemma of~\cite[Lemma~3]{marin}. If \(A\) denotes
the tensor field induced by \(A_0\), then
\(\nabla'=\nabla+\tfrac12A\) is infinitesimally homogeneous,
with
\begin{equation}
\label{eq:general-torsion-transformation}
T_0'(u,v)=T_0(u,v)+A_0(u,v),
\end{equation}
\begin{equation}
\label{eq:general-inner-transformation}
\mathfrak I_0'(u)=\lambda'(u)+\mathfrak g,
\qquad
\lambda'(u)=\lambda(u)+\tfrac12A_{0,u},
\end{equation}
and
\begin{equation}
\label{eq:general-curvature-deformation}
\begin{aligned}
R_0'(u,v)w={}&R_0(u,v)w
+\tfrac12\bigl(D_{\lambda(u)}A_0\bigr)(v,w)
-\tfrac12\bigl(D_{\lambda(v)}A_0\bigr)(u,w)\\
&+\tfrac12A_0\bigl(T_0(u,v),w\bigr)
+\tfrac14[A_{0,u},A_{0,v}]w,
\end{aligned}
\end{equation}
where \(A_{0,u}:=A_0(u,\cdot)\).

Since \(T_0\) is itself \(G\)-equivariant and skew-symmetric,
taking \(A_0=-T_0\) shows that the symmetrized connection
\(
\nabla^s=\nabla-\tfrac12T
\)
is torsion-free and that \((M,\nabla^s,P)\) is
infinitesimally homogeneous. Conversely, any torsion-free
infinitesimally homogeneous triple \((M,\nabla^s,P)\) and any
\(t_0\in\operatorname{Hom}_G(\Lambda^2\mathbb R^n,\mathbb R^n)\)
give rise to an infinitesimally homogeneous connection
\(\nabla=\nabla^s+\tfrac12t\), whose torsion is the tensor
field \(t\) induced by \(t_0\). Thus the classification with
arbitrary torsion reduces to the torsion-free case together
with the determination of
\(\operatorname{Hom}_G(\Lambda^2\mathbb R^n,\mathbb R^n)\).

\subsection{Some \(\operatorname{U}(n)\)-representation facts}
\label{subsec:unitary-representation-facts}

We now record the elementary representation-theoretic facts that
will be used in the analysis of \(G=\operatorname{U}(n)\times\{1\}\). Let
\(V=H_0\oplus\mathbb Re\), where \(H_0\simeq\mathbb R^{2n}\) is endowed
with its standard Euclidean metric \(h_0\), standard complex structure
\(J_0\), and fundamental two-form
\(\omega_0(X,Y)=h_0(J_0X,Y)\).
The group \(\operatorname{U}(n)\) acts on \(H_0\) through its standard
real representation and fixes \(e\); we denote by
\(G=\operatorname{U}(n)\times\{1\}\subset\operatorname{GL}(V)\) the
resulting group and by \(\mathfrak g=\mathfrak u(n)\oplus\{0\}\) its Lie
algebra. All representation spaces, tensor products and
\(\operatorname{Hom}\)- and \(\operatorname{End}\)-spaces in this
subsection are understood over \(\mathbb R\), unless otherwise stated.

\begin{lema}
\label{lem:unitary-facts}
The following hold for every \(n\geq1\):
\begin{enumerate}
\item[(a)] \(\operatorname{End}_{\operatorname{U}(n)}(H_0)=\mathbb R\,\mathrm{Id}_{H_0}\oplus \mathbb R\,J_0\);
\item[(b)] \(\bigl(H_0^*\otimes H_0^*\bigr)^{\operatorname{U}(n)}
=\mathbb R\,h_0\oplus\mathbb R\,\omega_0\); in particular
\(\bigl(\Lambda^2H_0^*\bigr)^{\operatorname{U}(n)}=\mathbb R\,\omega_0\).
Moreover, \(H_0^{\operatorname{U}(n)}=0\) and
\(\operatorname{Hom}_{\operatorname{U}(n)}(H_0,\mathbb R)=0\);
\item[(c)] \(\operatorname{Hom}_{\operatorname{U}(n)}(H_0\otimes H_0,H_0)=0\)
\textup{(}in particular
\(\operatorname{Hom}_{\operatorname{U}(n)}(\Lambda^2H_0,H_0)=0\)\textup{)},
\(\operatorname{Hom}_{\operatorname{U}(n)}\bigl(H_0,\mathfrak{gl}(H_0)\bigr)=0\), and
\(\operatorname{Hom}_{\operatorname{U}(n)}\bigl(H_0,\mathfrak{gl}(H_0)/\mathfrak u(n)\bigr)=0\);
\item[(d)] the fixed-point subspace of \(\operatorname{U}(n)\) in
\(\mathfrak{gl}(H_0)/\mathfrak u(n)\) is spanned by the class of
\(\mathrm{Id}_{H_0}\).
\end{enumerate}
\end{lema}

\begin{proof}
(a) Since \(J_0\in\operatorname{U}(n)\), every equivariant endomorphism
commutes with \(J_0\), hence is complex linear on \(H_0\simeq\mathbb C^n\).
Since \(\mathbb C^n\) is an irreducible complex representation of
\(\operatorname{U}(n)\), Schur's lemma shows that such an endomorphism is
multiplication by a complex number \(a+ib\), that is, the real
endomorphism \(a\,\mathrm{Id}_{H_0}+b\,J_0\).

(b) An invariant bilinear form can be written as \(h_0(L\cdot,\cdot)\)
with \(L\) equivariant; by (a), \(L=a\,\mathrm{Id}_{H_0}+b\,J_0\), which
gives \(a\,h_0+b\,\omega_0\). It is alternating if and only if \(a=0\).
The remaining assertions follow from the fact that
\(-\mathrm{Id}_{H_0}\in\operatorname{U}(n)\) acts by \(-1\) on \(H_0\)
and trivially on \(\mathbb R\).

(c) As in (b), \(-\mathrm{Id}_{H_0}\in\operatorname{U}(n)\) acts by
\(-1\) on \(H_0\), but trivially on \(H_0\otimes H_0\), on
\(\mathfrak{gl}(H_0)\) and on \(\mathfrak{gl}(H_0)/\mathfrak u(n)\).
Hence every equivariant map from \(H_0\otimes H_0\) to \(H_0\) vanishes,
and so does every equivariant map from \(H_0\) to
\(\mathfrak{gl}(H_0)\) or to \(\mathfrak{gl}(H_0)/\mathfrak u(n)\).

(d) The inner product \(\langle A,B\rangle=\operatorname{tr}(AB^{\mathsf T})\)
on \(\mathfrak{gl}(H_0)\) is invariant under conjugation by
\(\operatorname{U}(n)\subset\operatorname{O}(2n)\), and so is
\(\mathfrak u(n)\). Hence \(\mathfrak m:=\mathfrak u(n)^{\perp}\) is an
\(\operatorname{Ad}(\operatorname{U}(n))\)-invariant complement of
\(\mathfrak u(n)\), and the quotient map restricts to a
\(\operatorname{U}(n)\)-equivariant isomorphism
\(\mathfrak m\to\mathfrak{gl}(H_0)/\mathfrak u(n)\). Therefore the fixed
points of the quotient are the classes of the fixed points in
\(\mathfrak m\), that is, of
\(\mathfrak m\cap\operatorname{End}_{\operatorname{U}(n)}(H_0)\).
By (a), an element of this intersection has the form
\(a\,\mathrm{Id}_{H_0}+b\,J_0\). Since \(\mathrm{Id}_{H_0}\) is
orthogonal to every skew-symmetric endomorphism, it lies in
\(\mathfrak m\); since \(J_0\in\mathfrak u(n)\) and
\(\langle J_0,J_0\rangle=2n\neq0\), orthogonality to \(J_0\) forces
\(b=0\). Hence
\(\mathfrak m\cap\operatorname{End}_{\operatorname{U}(n)}(H_0)
=\mathbb R\,\mathrm{Id}_{H_0}\).
\end{proof}

The central element \(-\mathrm{Id}_{H_0}\in\operatorname{U}(n)\) also
gives a parity rule for \(G\)-equivariant maps on \(V\). Let
\[
\varepsilon=-\mathrm{Id}_{H_0}\oplus1\in G,
\]
which acts by \(-1\) on \(H_0\) and by \(+1\) on \(\mathbb Re\), and
let \(\mathfrak{gl}(V)=\mathfrak{gl}(V)_+\oplus\mathfrak{gl}(V)_-\) be
the decomposition into the \(\pm1\)-eigenspaces of
\(\operatorname{Ad}_\varepsilon\). With respect to
\(V=H_0\oplus\mathbb Re\), \(\mathfrak{gl}(V)_+\) consists of the
block-diagonal endomorphisms
\(\bigl(\begin{smallmatrix}B&0\\0&\rho\end{smallmatrix}\bigr)\), with
\(B\in\mathfrak{gl}(H_0)\) and \(\rho\in\mathbb R\), and
\(\mathfrak{gl}(V)_-\) of the off-diagonal ones
\(\bigl(\begin{smallmatrix}0&\mu\\\nu&0\end{smallmatrix}\bigr)\), with
\(\mu\in H_0\) and \(\nu\in H_0^*\). Note that
\(\mathfrak g\subset\mathfrak{gl}(V)_+\).

\begin{lema}[Parity]
\label{lem:parity}
Let \(F:V^k\to V\) be a \(G\)-equivariant multilinear map. If \(m\) of its arguments lie in \(H_0\) and the
remaining ones are multiples of \(e\), then the value of \(F\) lies in
\(H_0\) when \(m\) is odd and in \(\mathbb Re\) when \(m\) is even. In
particular, if \(T_0:\Lambda^2V\to V\), \(R_0:\Lambda^2V\to\mathfrak{gl}(V)\)
and \(\mathfrak I_0:V\to\mathfrak{gl}(V)/\mathfrak g\) are
\(G\)-equivariant and \(\lambda\) is a lift of \(\mathfrak I_0\), then,
for \(X,Y\in H_0\),
\[
T_0(X,Y)\in\mathbb Re,
\qquad
T_0(X,e)\in H_0,
\]
\[
R_0(X,Y)\in\mathfrak{gl}(V)_+,
\qquad
R_0(X,e)\in\mathfrak{gl}(V)_-,
\]
\[
\lambda(X)\in\mathfrak{gl}(V)_-+\mathfrak g,
\qquad
\lambda(e)\in\mathfrak{gl}(V)_+.
\]
\end{lema}

\begin{proof}
If \(m\) arguments of \(F\) are horizontal, equivariance with
\(g=\varepsilon\) gives
\[
(-1)^mF(u_1,\dots,u_k)=\varepsilon F(u_1,\dots,u_k),
\]
which is the first assertion; the statements on \(T_0\) are the cases
\(m=2,1\). Likewise, the identity
\(R_0(\varepsilon u,\varepsilon v)=\operatorname{Ad}_\varepsilon R_0(u,v)\)
gives the statements on \(R_0\).

For \(\mathfrak I_0\) the argument takes place in the quotient. Since
\(\operatorname{Ad}_\varepsilon\) preserves
\(\mathfrak g\subset\mathfrak{gl}(V)_+\), the \((\pm1)\)-eigenspaces of
the induced action on \(\mathfrak{gl}(V)/\mathfrak g\) are
\(\mathfrak{gl}(V)_+/\mathfrak g\) and the image of
\(\mathfrak{gl}(V)_-\). Equivariance places \(\mathfrak I_0(e)\) in the
first and \(\mathfrak I_0(X)\) in the second. Hence every lift satisfies
\(\lambda(X)\in\mathfrak{gl}(V)_-+\mathfrak g\), the
\(\mathfrak g\)-component reflecting the freedom in the lift, and
\(\lambda(e)\in\mathfrak{gl}(V)_++\mathfrak g=\mathfrak{gl}(V)_+\).
\end{proof}
For later use, we also recall the following curvature fact.

\begin{lema}[Invariant K\"ahler curvature tensors]
\label{lem:unitary-kahler-curvature}
For \(X,Y,Z\in H_0\), set
\begin{equation}
\label{eq:kahler-model-tensor}
K_{X,Y}Z
=
\frac14\Bigl(
h_0(Y,Z)X-h_0(X,Z)Y
+\omega_0(Y,Z)J_0X-\omega_0(X,Z)J_0Y
-2\,\omega_0(X,Y)J_0Z
\Bigr),
\end{equation}
and \(K(X,Y,Z,W)=h_0(K_{X,Y}Z,W)\). Then \(K_{X,Y}\in\mathfrak u(n)\),
and \(K\) is a \(\operatorname{U}(n)\)-invariant K\"ahler algebraic
curvature tensor normalized by
\begin{equation}
\label{eq:kahler-normalization}
K(X,J_0X,J_0X,X)=\|X\|^4,
\quad X\in H_0;
\end{equation}
that is, \(K\) is the curvature tensor of the complex space form of
constant holomorphic sectional curvature \(1\).

Let \(B\) be a \(\operatorname{U}(n)\)-equivariant linear map from
\(\Lambda^2H_0\) to \(\mathfrak u(n)\) satisfying the first Bianchi
identity
\[
\mathfrak S_{X,Y,Z}B(X,Y)Z=0.
\]
Then \(B=\kappa K\), where \(\kappa=h_0(B(X,J_0X)J_0X,X)\) for any unit
vector \(X\in H_0\).
\end{lema}

\begin{proof}
That \(K_{X,Y}\) is skew-adjoint and commutes with \(J_0\), that \(K\)
satisfies the first Bianchi identity, and that \(K\) is
\(\operatorname{U}(n)\)-invariant follow directly from
\eqref{eq:kahler-model-tensor}, since it is built from the invariant
tensors \(h_0\), \(\omega_0\) and \(J_0\). For a unit vector \(X\), one
computes \(K_{X,J_0X}J_0X=X\), which gives \eqref{eq:kahler-normalization}.

Now let \(B\) be as in the statement, and put
\(\mathcal R(X,Y,Z,W)=h_0(B(X,Y)Z,W)\). It is
\(\operatorname{U}(n)\)-invariant, skew-symmetric in \((X,Y)\) since
\(B\) is alternating, skew-symmetric in \((Z,W)\) since
\(\mathfrak u(n)\subset\mathfrak{so}(H_0)\), and satisfies the first
Bianchi identity. As is well known, these properties imply the pair
symmetry \(\mathcal R(X,Y,Z,W)=\mathcal R(Z,W,X,Y)\), so \(\mathcal R\)
is an algebraic curvature tensor. Since \(B(X,Y)\) commutes with
\(J_0\), it also satisfies the K\"ahler identity
\(\mathcal R(X,Y,J_0Z,J_0W)=\mathcal R(X,Y,Z,W)\).

For \(X\neq0\), consider the holomorphic sectional curvature of
\(\mathcal R\),
\[
H_{\mathcal R}(X)
=
\frac{\mathcal R(X,J_0X,J_0X,X)}{\|X\|^4}.
\]
Since \(\operatorname{U}(n)\) acts transitively on the unit sphere of
\(H_0\) and \(\mathcal R\) is \(\operatorname{U}(n)\)-invariant,
\(H_{\mathcal R}\equiv\kappa\) is constant. Hence
\(\mathcal R-\kappa K\) is a K\"ahler algebraic curvature tensor with
vanishing holomorphic sectional curvature. Since a K\"ahler algebraic
curvature tensor is determined by its holomorphic sectional curvature
\cite[Chapter~IX, Proposition~7.1]{KobayashiNomizuII}, we conclude
\(\mathcal R=\kappa K\), that is, \(B=\kappa K\).
\end{proof}

For \(n\geq2\), the condition that \(B\) take values in
\(\mathfrak u(n)\) is essential: \(B(X,Y)Z=h_0(Y,Z)X-h_0(X,Z)Y\) is
\(\operatorname{U}(n)\)-equivariant, \(\mathfrak{so}(H_0)\)-valued and
satisfies the first Bianchi identity, but it is not a multiple of
\(K\). For \(n=1\), \(\mathfrak u(1)=\mathfrak{so}(2)\) and the two
tensors coincide.

\subsection{Local equivalence and local homogeneity}
\label{subsec:local-homogeneity}

The classification in Section~\ref{sec:unitary-one} is carried out on
the characteristic tensors. To turn it into a statement about manifolds
we need two facts, which we prove here for compact \(G\): that the
Levi-Civita connection of the metric determined by \(P\) is again
infinitesimally homogeneous, so that every infinitesimally homogeneous
connection is the Levi-Civita connection plus an invariant tensor; and
that two triples with the same characteristic tensors are locally
equivalent.

In this subsection \(G\) is assumed to be compact. Then \(\mathfrak g\)
admits an \(\operatorname{Ad}(G)\)-invariant complement
\(\mathfrak m\subset\mathfrak{gl}(\mathbb R^n)\) (when
\(G\subset\operatorname{O}(n)\), as in the case of interest here, one
may take the orthogonal complement with respect to
\(\langle A,B\rangle=\operatorname{tr}(AB^{\mathsf T})\), as in the
proof of Lemma~\ref{lem:unitary-facts}\textup{(d)}). The quotient map
restricts to a \(G\)-equivariant isomorphism
\(\pi|_{\mathfrak m}:\mathfrak m\to\mathfrak{gl}(\mathbb R^n)/\mathfrak g\),
so every \(G\)-equivariant \(\mathfrak I_0\) admits the lift
\(\lambda=(\pi|_{\mathfrak m})^{-1}\circ\mathfrak I_0\), which is
\(G\)-equivariant on the nose: \(\lambda(gu)=\operatorname{Ad}_g\lambda(u)\).
For such a lift, \(\tilde\lambda_{pu}(pv):=p\,\lambda(u)v\) defines a
tensor field \(\tilde\lambda\) on \(M\).

\begin{prop}\label{prop:local-homogeneity}
Let \((M,\nabla,P)\) be infinitesimally homogeneous with \(G\) compact,
and let \(\lambda\) be an equivariant lift of \(\mathfrak I_0\).
\begin{enumerate}
\item[(a)] The affine connection \(\nabla^P:=\nabla-\tilde\lambda\) is
compatible with \(P\), that is, it preserves the \(G\)-structure. Its torsion \(T^P\) and curvature \(R^P\), as well as \(\tilde\lambda\),
are \(\nabla^P\)-parallel.

\item[(b)] For \(x,y\in M\), \(p\in P_x\) and \(q\in P_y\) there exist
neighbourhoods \(x\in U\), \(y\in U'\) and a diffeomorphism
\(f:U\to U'\) with \(f(x)=y\), \(df_x\circ p=q\),
\(f(P|_U)=P|_{U'}\) and \(f^*\nabla=\nabla\).

\item[(c)] Every tensor field on \(M\), and more generally every section
of a bundle associated with \(P\), that is invariant under a family
of local diffeomorphisms preserving \(P\) and acting transitively on
its frames, such as the family in \textup{(b)}, has constant components
in the frames of \(P\). In particular, applying this to the family in
\textup{(b)}, if \(G\subset\operatorname{O}(n)\)
and \(g\) denotes the metric determined by \(P\), then
\((M,\nabla^g,P)\) is infinitesimally homogeneous for the Levi-Civita
connection \(\nabla^g\), and \(\nabla-\nabla^g\) is induced by an
element of
\(\operatorname{Hom}_{G}(\mathbb R^n\otimes\mathbb R^n,\mathbb R^n)\).

\item[(d)] Two infinitesimally homogeneous triples with structure group
\(G\) and the same characteristic tensors are locally isomorphic.
\end{enumerate}
\end{prop}

\begin{proof}
(a) In a local section \(s\) of \(P\) the connection form of
\(\nabla^P\) is \(\Lambda^s-\lambda\), which takes values in
\(\mathfrak g\); hence \(\nabla^P\) preserves \(P\). By
\eqref{eq:S-deformation-TI} and \eqref{eq:S-deformation-R} with
\(S_0(u,v)=-\lambda(u)v\), the tensors \(T^P\) and \(R^P\) have
constant \(G\)-invariant components in the frames of \(P\), and so does
\(\tilde\lambda\). A tensor field with constant \(G\)-invariant components in the frames
of \(P\) is parallel with respect to every connection preserving \(P\):
in a local section of \(P\), its covariant derivative is represented by
\(D_L\) of its components, with \(L\in\mathfrak g\), which vanishes.

(b) Put \(F=q\circ p^{-1}:T_xM\to T_yM\). Since \(T^P\) and \(R^P\) have
the same constant components in \(p\) and in \(q\), \(F\) maps
\(T^P_x,R^P_x\) to \(T^P_y,R^P_y\). Since moreover \(T^P\) and \(R^P\)
are \(\nabla^P\)-parallel, \cite[Chapter~VI, Theorem~7.4]{KobayashiNomizuI}
gives an affine diffeomorphism \(f\) of \(\nabla^P\) between
neighbourhoods of \(x\) and \(y\) with \(df_x=F\). Since \(f\) commutes with \(\nabla^P\)-parallel
transport along radial geodesics, and \(P\) is invariant under that
parallel transport, \(f\) maps \(P\) to \(P\). Consequently \(f\)
preserves \(\tilde\lambda\), which has constant components in \(P\),
and therefore
\(f^*\nabla=f^*(\nabla^P+\tilde\lambda)=\nabla^P+\tilde\lambda=\nabla\).

(c) A tensor field invariant under a family of local diffeomorphisms
preserving \(P\) and acting transitively on its frames has the same
components in all frames of \(P\). The same argument applies to
sections of bundles associated with \(P\); in particular to the inner
torsion, which is a section of the bundle with fibres
\(\operatorname{Hom}\bigl(T_xM,\mathfrak{gl}(T_xM)/\mathfrak g_x\bigr)\),
on which these diffeomorphisms act through their differentials. The
diffeomorphisms of (b) preserve \(P\), hence \(g\), hence \(\nabla^g\).
So the curvature and inner torsion of \(\nabla^g\), and the difference
tensor \(S:=\nabla-\nabla^g\), are invariant. In particular,
\(S_x(pu,pv)=p\,S_0(u,v)\) for a fixed bilinear map \(S_0\) and every
frame \(p\in P\). Replacing \(p\) by \(pa\), with \(a\in G\), gives
\(p\,S_0(au,av)=S_x(pau,pav)=pa\,S_0(u,v)\); hence
\(S_0(au,av)=aS_0(u,v)\), that is,
\(S_0\in\operatorname{Hom}_G(\mathbb R^n\otimes\mathbb R^n,\mathbb R^n)\).

(d) Apply \cite[Chapter~VI, Theorem~7.4]{KobayashiNomizuI} to the
connections \(\nabla^P\) of the two triples, built with the same
equivariant lift, and argue as in (b).
\end{proof}

For \(G=\{1\}\), \(P\) is an absolute parallelism with constant
structure functions, \(\nabla^P\) is its flat connection, and
Proposition~\ref{prop:local-homogeneity} reduces to the classical fact
that \(M\) is then locally a Lie group, with \(\nabla\) a
left-invariant connection. In general, the proof reduces the statement
to the classical theorem \cite[Chapter~VI, Theorem~7.4]{KobayashiNomizuI}
applied to the auxiliary connection \(\nabla^P\), and the proposition
may be viewed as the order-zero counterpart, relative to \(P\), of the
theorems of Singer~\cite{Singer} and Opozda~\cite{Opozda98}.

%%%%%%%%%%%%%%%%%%%%%%%
%%%%%%%%%%%%%%%%%%%%%%

\section{Infinitesimally homogeneous structures for the symplectic group
\(\operatorname{Sp}(2n,\mathbb R)\)}
\label{sec:symplectic}

In this section, \(V=\mathbb R^{2n}\) is endowed with its standard symplectic form
\(\omega_0\). We denote by
\(\operatorname{Sp}(V,\omega_0)=\operatorname{Sp}(2n,\mathbb R)\) the group
of linear automorphisms of \(V\) preserving \(\omega_0\), and by
\(\mathfrak{sp}(V,\omega_0)\) its Lie algebra.

\subsection{Geometric interpretation of the \(G\)-structure}
\label{subsec:symplectic-geometric-interpretation}

An \(\operatorname{Sp}(V,\omega_0)\)-structure on a \(2n\)-dimensional
manifold \(M\) is equivalent to an almost symplectic structure, that is, to the choice of a nondegenerate
\(2\)-form \(\omega\) on \(M\).

Indeed, if \(P\subseteq \operatorname{FR}(TM)\) is an
\(\operatorname{Sp}(V,\omega_0)\)-structure, we define \(\omega_x\) by
\[
\omega_x(pu,pv)=\omega_0(u,v),
\quad p\in P_x,\quad u,v\in V.
\]
This is well defined because the transition maps between frames in
\(P\) belong to \(\operatorname{Sp}(V,\omega_0)\), and the action of
\(\operatorname{Sp}(V,\omega_0)\) preserves \(\omega_0\). Conversely, a
nondegenerate \(2\)-form \(\omega\) determines the bundle of symplectic
frames
\[
P=\left\{p:V\to T_xM:\ p^*\omega_x=\omega_0\right\}.
\]

At this stage, \(\omega\) is not assumed to be closed. Its closedness
will follow from the infinitesimal homogeneity conditions.

\subsection{Characteristic tensors associated with the symplectic group}
\label{subsec:symplectic-characteristic-tensors}
Let \((M,\nabla,P)\) be an infinitesimally homogeneous manifold with
structure group \(\operatorname{Sp}(2n,\mathbb R)\). 
Since \(-I_V\in\operatorname{Sp}(2n,\mathbb R)\) acts by \(-1\) on \(V\)
and trivially on \(\Lambda^2V\) and on
\(\mathfrak{gl}(V)/\mathfrak{sp}(V,\omega_0)\), equivariance with
\(g=-I_V\) gives \(T_0=-T_0\) and \(\mathfrak I_0=-\mathfrak I_0\)
\textup{(}cf.~\cite[Lemma~3.1]{Marin-Blazquez}\textup{)}; hence
\[
T_0=0,\qquad \mathfrak I_0=0.
\]

The condition \(\mathfrak I_0=0\) means that \(\nabla\) is compatible with
the \(\operatorname{Sp}(2n,\mathbb R)\)-structure. Equivalently,
\(\nabla\omega=0\). Moreover, since \(T_0=0\) and \(\mathfrak I_0=0\), we
may take \(\lambda=0\) in the torsion-free algebraic relation
\eqref{eq:torsion-free-compatibility}. Hence
\[
R_0(u,v)\in\mathfrak{sp}(V,\omega_0),
\quad u,v\in V.
\]
Therefore the curvature characteristic tensor may be regarded as an
\(\operatorname{Sp}(V,\omega_0)\)-equivariant linear map
\(R_0:\Lambda^2V\to\mathfrak{sp}(V,\omega_0)\).

\begin{lema}
\label{lem:no-sp-equivariant-curvature}
There is no nonzero \(\operatorname{Sp}(V,\omega_0)\)-equivariant linear
map \(\Lambda^2V\to\mathfrak{sp}(V,\omega_0)\). Equivalently,
\[
\operatorname{Hom}_{\operatorname{Sp}(V,\omega_0)}
\bigl(\Lambda^2V,\mathfrak{sp}(V,\omega_0)\bigr)=0.
\]
\end{lema}

\begin{proof}
Let \(F:\Lambda^2V\to\mathfrak{sp}(V,\omega_0)\) be an \(\operatorname{Sp}(V,\omega_0)\)-equivariant linear map. Its image is an
\(\operatorname{Ad}\)-invariant subspace of \(\mathfrak{sp}(V,\omega_0)\);
differentiating at the identity, it is invariant under
\(\operatorname{ad}\bigl(\mathfrak{sp}(V,\omega_0)\bigr)\), that is, it
is an ideal. As \(\mathfrak{sp}(2n,\mathbb R)\) is simple for every \(n\geq1\),
the image is either \(0\) or all of \(\mathfrak{sp}(V,\omega_0)\). The
second possibility is excluded by dimension:
\[
\dim\Lambda^2V=n(2n-1)<n(2n+1)=\dim\mathfrak{sp}(V,\omega_0).
\]
Hence \(F=0\).
\end{proof}

\begin{teo}
\label{thm:symplectic-classification}
Let \((M,\nabla,P)\) be an infinitesimally homogeneous manifold with
structure group \(\operatorname{Sp}(2n,\mathbb R)\). Then \(P\) is the
bundle of symplectic frames of a symplectic form \(\omega\) on \(M\),
and \(\nabla\) is a flat torsion-free symplectic connection.
Equivalently,
\[
T=0,
\quad
R=0,
\quad
\nabla\omega=0.
\]

Conversely, let \((M,\omega)\) be a symplectic manifold endowed with a
flat torsion-free connection \(\nabla\) such that \(\nabla\omega=0\). If
\(P\) denotes the bundle of symplectic frames of \(\omega\), then
\((M,\nabla,P)\) is infinitesimally homogeneous with structure group
\(\operatorname{Sp}(2n,\mathbb R)\).
\end{teo}

\begin{proof}
By the discussion above and Lemma~\ref{lem:no-sp-equivariant-curvature},
\(T_0=0\), \(\mathfrak I_0=0\) and \(R_0=0\); that is, \(T=0\), \(R=0\)
and \(\nabla\omega=0\). Since \(\nabla\) is torsion-free and \(\omega\)
is parallel, \(d\omega=0\). Conversely, if \(\nabla\) is flat,
torsion-free and symplectic, it is compatible with the bundle \(P\) of
symplectic frames, so all three characteristic tensors exist and
vanish.
\end{proof}

\begin{remark}[Local model]
\label{rem:symplectic-local-model}
The structures of Theorem~\ref{thm:symplectic-classification} are all
locally equivalent to the standard model
\((\mathbb R^{2n},\omega_0,\partial)\), where \(\partial\) denotes the
canonical flat connection. Indeed, since \(\nabla\) is flat and
torsion-free, every point has a neighbourhood with affine coordinates
\((x^1,\dots,x^{2n})\) such that \(\nabla\partial_i=0\). Since \(\omega\)
is parallel, its coefficients \(\omega(\partial_i,\partial_j)\) are
constant, and a linear change of affine coordinates brings \(\omega\)
to the standard form \(\omega_0\).
\end{remark}

\begin{remark}[Non-flat symplectic connections]
\label{rem:special-symplectic}
Theorem~\ref{thm:symplectic-classification} is not a restriction on symplectic
connections in general: every symplectic manifold admits torsion-free
symplectic connections, and these are generically non-flat. What it
expresses is that, by Lemma~\ref{lem:no-sp-equivariant-curvature}, a
nonzero curvature tensor cannot have the same constant representative
in all symplectic frames. Hence a non-flat torsion-free symplectic
connection can be infinitesimally homogeneous only with respect to a
reduction of the symplectic frame bundle to a proper subgroup of
\(\operatorname{Sp}(2n,\mathbb R)\). This answers the question about
the symplectic group raised in the concluding remarks
of~\cite{Marin-Blazquez}.
\end{remark}

%%%%%%%%%%%%%%%%%%%%%%%%
%%%%%%%%%%%%%%%%%%%%%%

\section{Infinitesimally homogeneous structures for the group
\(\operatorname{U}(n)\times\{1\}\)}
\label{sec:unitary-one}

Let \(V=H_0\oplus \mathbb R e\), where \(H_0\simeq \mathbb R^{2n}\) and
\(e=e_{2n+1}\), with \(h_0\), \(J_0\), \(\omega_0\) as in
Subsection~\ref{subsec:unitary-representation-facts}. We consider the
Lie subgroup \(G=\operatorname{U}(n)\times\{1\}\subset \operatorname{GL}(V)\),
\[
G=
\left\{
\begin{pmatrix}
A&0\\
0&1
\end{pmatrix}
:
A\in \operatorname{U}(n)
\right\}
\cong \operatorname{U}(n).
\]
Thus \(G\) acts on the horizontal model space \(H_0\) through the
standard unitary representation and fixes the vertical vector \(e\).

\subsection{Geometric interpretation of the \(G\)-structure}
\label{subsec:unitary-geometric-interpretation}

A \(G\)-structure on a \((2n+1)\)-dimensional manifold \(M\) is
equivalent to the following geometric data:
\(TM=H\oplus\mathbb R\xi\), where \(H\subset TM\) is a rank \(2n\)
distribution, \(\xi\) is a distinguished nowhere vanishing vector field,
and \(H\) is endowed with a Hermitian structure. More explicitly, \(H\)
carries a Riemannian metric \(h\) and an \(h\)-orthogonal almost complex
structure \(J:H\to H\), \(J^2=-\operatorname{Id}_H\).

Indeed, given such data, the corresponding \(G\)-structure consists of
the frames \(p:V\to T_xM\) satisfying \(p(H_0)=H_x\), \(p(e)=\xi_x\),
and such that \(p|_{H_0}:(H_0,h_0,J_0)\to(H_x,h_x,J_x)\) is unitary.
Conversely, a \(G\)-structure determines precisely such a splitting, a
distinguished vector field, and a Hermitian structure on the horizontal
distribution. The stabilizer of the model data \((H_0,h_0,J_0,e)\)
inside \(\operatorname{GL}(V)\) is precisely
\(\operatorname{U}(n)\times\{1\}\).

Equivalently, define a \(1\)-form \(\eta\) by \(\eta(\xi)=1\),
\(\eta|_H=0\), and extend \(h\) to a Riemannian metric \(g\) on \(TM\) by
declaring \(\xi\) to have unit length and to be orthogonal to \(H\).
Thus \(g=h+\eta\otimes\eta\). Define an endomorphism
\(\phi:TM\to TM\) by \(\phi|_H=J\), \(\phi(\xi)=0\). Then
\((\phi,\xi,\eta,g)\) is an almost contact metric structure
\cite{Blair}:
\[
\eta(\xi)=1,
\quad
\ker\eta=H,
\quad
\phi^2=-\operatorname{Id}+\eta\otimes\xi,
\quad
g(\phi X,\phi Y)=g(X,Y)-\eta(X)\eta(Y),
\quad X,Y\in TM.
\]
Conversely, an almost contact metric structure \((\phi,\xi,\eta,g)\)
gives back \(H=\ker\eta\), \(h=g|_H\) and \(J=\phi|_H\): since
\(\eta\circ\phi=0\) and \(g(X,\xi)=\eta(X)\), \(\phi\) preserves \(H\)
and \(\xi\) is a unit vector orthogonal to \(H\). The two constructions
are mutually inverse, so \(G\)-structures and almost contact metric
structures are equivalent.
We also write \(\omega(X,Y)=g(\phi X,Y)\); thus \(\omega\) restricts to
\(h(JX,Y)\) on \(H\) and \(\iota_\xi\omega=0\). Note that the fundamental
\(2\)-form of \cite{Blair}, \(\Phi(X,Y)=g(X,\phi Y)\), equals
\(-\omega\). Throughout, the exterior derivative of a \(1\)-form is
normalized by \(d\eta(X,Y)=X\eta(Y)-Y\eta(X)-\eta([X,Y])\).

No differential condition is imposed at this stage. Thus the
\(G\)-structure determines only the underlying almost contact metric
data. In particular, \(H=\ker\eta\) is not assumed to be integrable or
contact, and the almost contact metric structure is not assumed to be
normal.

\subsection{Invariant candidates for the characteristic tensors}
\label{subsec:unitary-invariant-candidates}

Here we impose only \(G\)-equivariance. The conditions
for these candidates to occur as characteristic tensors of an
infinitesimally homogeneous structure, namely the compatibility
relation and the Bianchi identities of
Theorem~\ref{thm:algebraic-characterization}, are studied in
Subsection~\ref{subsec:unitary-torsion-free}.

\subsubsection{Invariant torsion tensors}
\label{subsubsec:unitary-torsion}

We first determine the possible \(G\)-equivariant candidates for the
torsion characteristic tensor \(T_0:\Lambda^2V\to V\). Since \(G\)
preserves \(H_0\) and fixes \(e\), we have the \(G\)-equivariant
decomposition
\(\Lambda^2V=\Lambda^2H_0\oplus(H_0\wedge e)\).

\begin{prop}
\label{prop:invariant-torsion-Un}
Every \(G\)-equivariant skew-symmetric bilinear map
\(T_0:V\times V\to V\) is of the form
\begin{equation}
\label{eq:torsion-normal-form-Un}
T_0(X,Y)=\tau_3\,\omega_0(X,Y)e,
\quad
T_0(X,e)=\tau_1X+\tau_2J_0X,
\quad
X,Y\in H_0,
\end{equation}
where \(\tau_1,\tau_2,\tau_3\in\mathbb R\) are uniquely determined by
\(T_0\), and the remaining values are determined by skew-symmetry.
Conversely, every map of this form is \(G\)-equivariant; thus
\(\operatorname{Hom}_G(\Lambda^2V,V)\) is three-dimensional.
\end{prop}

\begin{proof}
By Lemma~\ref{lem:parity}, \(T_0(\Lambda^2H_0)\subset\mathbb Re\) and
\(T_0(H_0\wedge e)\subset H_0\). The map sending \((X,Y)\) to the \(e\)-component of \(T_0(X,Y)\) is a \(\operatorname{U}(n)\)-invariant alternating
\(2\)-form on \(H_0\), hence a multiple of \(\omega_0\) by
Lemma~\ref{lem:unitary-facts}(b). The map \(X\mapsto T_0(X,e)\) is a
\(\operatorname{U}(n)\)-equivariant endomorphism of \(H_0\), hence of
the form \(\tau_1X+\tau_2J_0X\) by Lemma~\ref{lem:unitary-facts}(a).
The parameters are unique because \(\omega_0\neq0\) and \(X\), \(J_0X\)
are linearly independent for \(X\neq0\). Conversely, a map of this form
is \(G\)-equivariant, since \(\omega_0\), \(\mathrm{Id}_{H_0}\) and
\(J_0\) are \(\operatorname{U}(n)\)-invariant and \(e\) is fixed by
\(G\).
\end{proof}

\begin{remark}
\label{rem:geometric-meaning-torsion-parameters}
Under the geometric identification above, the model space \(H_0\)
corresponds to \(H\), while \(e\) corresponds to \(\xi\). Thus, for
\(X,Y\in H\),
\[
T(X,Y)=\tau_3\,\omega(X,Y)\xi,
\quad
T(X,\xi)=\tau_1X+\tau_2JX.
\]
The parameter \(\tau_3\) measures the vertical component of the torsion
of two horizontal vectors, whereas \(\tau_1,\tau_2\) describe the mixed
torsion involving the distinguished direction. The relation between
\(\tau_3\), the inner torsion and \(d\eta\) is given in
Corollary~\ref{cor:invariants}.
\end{remark}

\subsubsection{Invariant inner torsion tensors}
\label{subsubsec:unitary-inner-torsion}

We next determine the possible \(G\)-equivariant candidates for the
inner torsion characteristic tensor
\(\mathfrak I_0:V\to\mathfrak{gl}(V)/\mathfrak g\). Choose a lift
\(\lambda:V\to\mathfrak{gl}(V)\) of \(\mathfrak I_0\). With respect to the
decomposition \(V=H_0\oplus\mathbb Re\), write
\[
\lambda(u)
=
\begin{pmatrix}
B(u)&\mu(u)\\
\nu(u)&\rho(u)
\end{pmatrix},
\quad
B(u)\in\mathfrak{gl}(H_0),
\quad
\mu(u)\in H_0,
\quad
\nu(u)\in H_0^*,
\quad
\rho(u)\in\mathbb R.
\]

By Lemma~\ref{lem:parity}, \(\mu(e)=0\), \(\nu(e)=0\), and, for
\(X\in H_0\), \(B(X)\in\mathfrak u(n)\) and \(\rho(X)=0\). After
subtracting from \(\lambda\) the \(\mathfrak g\)-valued linear map
that sends \(X\in H_0\) to \(B(X)\oplus0\) and \(e\) to \(0\), we may
therefore assume \(B(X)=0\).

The equivariance of \(\mathfrak I_0\) means that
\(\operatorname{Ad}_g\lambda(u)-\lambda(gu)\in\mathfrak g\) for every
\(g=\operatorname{diag}(A,1)\) with \(A\in\operatorname{U}(n)\).

For \(u=e\) we have \(ge=e\), and the condition reads
\(AB(e)A^{-1}-B(e)\in\mathfrak u(n)\). Thus the class of \(B(e)\) in
\(\mathfrak{gl}(H_0)/\mathfrak u(n)\) is fixed by \(\operatorname{U}(n)\),
and by Lemma~\ref{lem:unitary-facts}(d), after modifying the lift by a
\(\mathfrak g\)-valued map, we may assume
\(B(e)=a\,\operatorname{Id}_{H_0}\). The vertical scalar component is
unrestricted, and we write \(\rho(e)=b\).

For \(u=X\in H_0\), since \(\mathfrak g\) has no off-diagonal
components, the off-diagonal blocks of the condition give
\[
A\,\mu(X)=\mu(AX),
\qquad
\nu(X)\circ A^{-1}=\nu(AX),
\qquad A\in\operatorname{U}(n).
\]
Write \(\nu(X)=\Qend(X)^\flat=h_0(\Qend(X),\cdot)\). Since \(A\) is
orthogonal, the second condition reads
\(\Qend(AX)=A\,\Qend(X)\). Thus \(\mu\) and \(\Qend\) are
\(\operatorname{U}(n)\)-equivariant endomorphisms of \(H_0\), and by
Lemma~\ref{lem:unitary-facts}(a) there exist
\(\alpha,\beta,\gamma,\delta\in\mathbb R\) such that
\[
\mu(X)=\Pend(X):=\alpha X+\beta J_0X,
\qquad
\Qend(X)=\gamma X+\delta J_0X.
\]

We have proved the following normal form.

\begin{prop}
\label{prop:inner-torsion-Un}
Every \(G\)-equivariant map \(\mathfrak I_0:V\to\mathfrak{gl}(V)/\mathfrak g\)
admits a lift \(\lambda:V\to\mathfrak{gl}(V)\) of the form
\begin{equation}
\label{eq:inner-torsion-lift}
\lambda(X)
=
\begin{pmatrix}
0&\alpha X+\beta J_0X\\[1mm]
(\gamma X+\delta J_0X)^\flat&0
\end{pmatrix},
\qquad
\lambda(e)
=
\begin{pmatrix}
a\operatorname{Id}_{H_0}&0\\
0&b
\end{pmatrix},
\end{equation}
where \(X\in H_0\) and \(a,b,\alpha,\beta,\gamma,\delta\in\mathbb R\).
Thus \(\mathfrak I_0(u)=\lambda(u)+\mathfrak g\) for \(u\in V\). Conversely, every map of the form \eqref{eq:inner-torsion-lift} defines
a \(G\)-equivariant map \(\mathfrak I_0:V\to\mathfrak{gl}(V)/\mathfrak g\).  
\end{prop}

\begin{remark}
\label{rem:inner-torsion-parameters}
The six parameters \(a,\ b,\ \alpha,\ \beta,\ \gamma,\ \delta\) are
uniquely determined by \(\mathfrak I_0\), since
\(\operatorname{Id}_{H_0}\notin\mathfrak u(n)\) and \(\mathfrak g\) has
no off-diagonal or vertical components. The parameters \(a,b\) arise
from the value of \(\mathfrak I_0\) on the distinguished direction \(e\),
while \(\alpha,\beta,\gamma,\delta\) describe the two
horizontal--vertical off-diagonal components. In particular,
\(\mathfrak I_0=0\) if and only if
\(a=b=\alpha=\beta=\gamma=\delta=0\). Moreover, the lift \eqref{eq:inner-torsion-lift} is \(G\)-equivariant on
the nose, \(\lambda(gu)=\operatorname{Ad}_g\lambda(u)\), since the
endomorphisms \(\Pend\) and \(\Qend\) are combinations of
\(\mathrm{Id}_{H_0}\) and \(J_0\), and hence commute with
\(\operatorname{U}(n)\); it is therefore an equivariant lift in the sense of
Subsection~\ref{subsec:local-homogeneity}.
\end{remark}

\subsubsection{Invariant curvature tensors}
\label{subsubsec:unitary-curvature}

We now determine the possible \(G\)-equivariant candidates for the
curvature characteristic tensor \(R_0:\Lambda^2V\to\mathfrak{gl}(V)\).
Write
\[
R_0(u,v)
=
\begin{pmatrix}
B(u,v)&\mu(u,v)\\
\nu(u,v)&r(u,v)
\end{pmatrix}.
\]
By Lemma~\ref{lem:parity}, \(R_0(X,e)\in\mathfrak{gl}(V)_-\) and
\(R_0(X,Y)\in\mathfrak{gl}(V)_+\) for \(X,Y\in H_0\), that is,
\[
B(X,e)=0,\quad r(X,e)=0,
\quad
\mu(X,Y)=0,\quad \nu(X,Y)=0.
\]
The equivariance condition
\(R_0(gu,gv)=\operatorname{Ad}_gR_0(u,v)\), with
\(g=\operatorname{diag}(A,1)\), \(A\in\operatorname{U}(n)\), and
\(ge=e\), gives for the remaining components
\[
\mu(AX,e)=A\,\mu(X,e),
\qquad
\nu(AX,e)=\nu(X,e)\circ A^{-1},
\]
\[
r(AX,AY)=r(X,Y),
\qquad
B(AX,AY)=A\,B(X,Y)A^{-1}.
\]
By Lemma~\ref{lem:unitary-facts}(a), applied to \(\mu(\cdot,e)\) and,
as in Subsection~\ref{subsubsec:unitary-inner-torsion}, to the
endomorphism dual to \(\nu(\cdot,e)\), the first two relations give
constants \(c_1,c_2,c_3,c_4\in\mathbb R\) such that
\begin{equation}
\label{eq:curvature-mixed}
R_0(X,e)
=
\begin{pmatrix}
0&c_1X+c_2J_0X\\[1mm]
(c_3X+c_4J_0X)^\flat&0
\end{pmatrix}.
\end{equation}
By Lemma~\ref{lem:unitary-facts}(b), \(r(X,Y)=c_0\,\omega_0(X,Y)\) for
some \(c_0\in\mathbb R\). Finally,
\(B:\Lambda^2H_0\to\mathfrak{gl}(H_0)\) is
\(\operatorname{U}(n)\)-equivariant. Thus
\begin{equation}
\label{eq:curvature-horizontal}
R_0(X,Y)
=
\begin{pmatrix}
B(X,Y)&0\\
0&c_0\,\omega_0(X,Y)
\end{pmatrix}.
\end{equation}

We summarize this discussion as follows.

\begin{prop}
\label{prop:curvature-candidates-Un}
Let \(R_0:\Lambda^2V\to\mathfrak{gl}(V)\) be a \(G\)-equivariant
candidate for the curvature characteristic tensor. Then \(R_0\) is
uniquely determined by
\[
B\in\operatorname{Hom}_{\operatorname{U}(n)}\bigl(\Lambda^2H_0,\mathfrak{gl}(H_0)\bigr)
\]
and by constants \(c_0,c_1,c_2,c_3,c_4\in\mathbb R\) through
\eqref{eq:curvature-mixed} and \eqref{eq:curvature-horizontal}.
Conversely, every tensor defined by these expressions is
\(G\)-equivariant.
\end{prop}

\subsection{Classification of the torsion-free stratum}
\label{subsec:unitary-torsion-free}

\subsubsection{Compatibility relation and Bianchi identities}
\label{subsubsec:unitary-torsion-free-relations}

Assume throughout this subsection that \(T_0=0\). Let
\(\lambda:V\to\mathfrak{gl}(V)\) be a lift of \(\mathfrak I_0\) in the
normal form \eqref{eq:inner-torsion-lift}. For \(X\in H_0\), write
\[
\Pend(X)=\alpha X+\beta J_0X,
\quad
\Qend(X)=\gamma X+\delta J_0X.
\]
Then
\[
\lambda(X)
=
\begin{pmatrix}
0&\Pend(X)\\
\Qend(X)^\flat&0
\end{pmatrix},
\quad
\lambda(e)
=
\begin{pmatrix}
a\,\operatorname{Id}_{H_0}&0\\
0&b
\end{pmatrix}.
\]

By Theorem~\ref{thm:algebraic-characterization}, the \(G\)-equivariant
maps \(R_0,\mathfrak I_0\) occur as characteristic tensors precisely when
they satisfy the compatibility relation
\eqref{eq:torsion-free-compatibility} together with the first and
second Bianchi identities.

Applying \eqref{eq:torsion-free-compatibility} to the mixed pair
\((X,e)\) yields
\begin{equation}
\label{eq:mixed-curvature-coefficients-torsion-free}
\begin{aligned}
c_1&=-\alpha^2+b\alpha+\beta^2,
&
c_2&=\beta(b-2\alpha),\\
c_3&=(2a-\alpha-b)\gamma+\beta\delta,
&
c_4&=(2a-\alpha-b)\delta-\beta\gamma.
\end{aligned}
\end{equation}

For the horizontal pair \((X,Y)\), one obtains
\begin{equation}
\label{eq:c0-torsion-free}
c_0=2(\alpha\delta-\beta\gamma-b\delta),
\end{equation}
and
\begin{equation}
\label{eq:B-congruence-torsion-free}
B(X,Y)-\bigl(\Pend(X)\otimes \Qend(Y)^\flat-\Pend(Y)\otimes \Qend(X)^\flat\bigr)
+2a\delta\,\omega_0(X,Y)\operatorname{Id}_{H_0}
\in\mathfrak u(n).
\end{equation}

Regarding the first Bianchi identity, since the cyclic sum in \eqref{eq:first-bianchi} is alternating, it
suffices to consider triples of horizontal vectors and triples of the
form \((X,Y,e)\), with \(X,Y\in H_0\). For horizontal \(X,Y,Z\), the
identity reads
\begin{equation}
\label{eq:B-first-bianchi}
\mathfrak S_{X,Y,Z}B(X,Y)Z=0.
\end{equation}
For a mixed triple, the curvature blocks give
\[
R_0(X,Y)e+R_0(Y,e)X+R_0(e,X)Y=(c_0-2c_4)\,\omega_0(X,Y)\,e,
\]
since the terms involving \(c_3\) cancel by the symmetry of \(h_0\).
Thus the mixed identity is equivalent to
\begin{equation}
\label{eq:c0-equals-2c4}
c_0=2c_4.
\end{equation}
By \eqref{eq:mixed-curvature-coefficients-torsion-free} and
\eqref{eq:c0-torsion-free}, we have \(c_0-2c_4=4\delta(\alpha-a)\).
Hence
\begin{equation}
\label{eq:delta-alpha-a}
\delta(\alpha-a)=0.
\end{equation}

We next determine the horizontal block in each branch and then apply
the second Bianchi identity \eqref{eq:second-bianchi}.

\subsubsection{The compatible torsion-free family}
\label{subsubsec:compatible-torsion-free}

\begin{teo}
\label{thm:compatible-torsion-free}
Let \(G=\operatorname{U}(n)\times\{1\}\), and let \((M,\nabla,P)\) be an
infinitesimally homogeneous affine manifold with \(G\)-structure
satisfying \(T_0=\mathfrak I_0=0\). Then \(\nabla\) is the Levi-Civita
connection of the metric \(g\) determined by \(P\). Locally, \(M\) is a
Riemannian product \(N^{2n}\times I\), where \(N\) is K\"ahler with
constant holomorphic sectional curvature \(\kappa\), and \(I\) is an
interval whose unit coordinate field corresponds to \(\xi\). In adapted
frames,
\[
R_0(X,e)=0,\qquad R_0(X,Y)e=0,\qquad R_0(X,Y)Z=\kappa K_{X,Y}Z,
\qquad X,Y,Z\in H_0.
\]

Conversely, such a product, endowed with its induced \(G\)-structure
and Levi-Civita connection, is infinitesimally homogeneous with
\(T_0=\mathfrak I_0=0\).
\end{teo}

\begin{proof}
Since \(\mathfrak I_0=0\), the connection \(\nabla\) preserves
\(g\), \(\xi\), \(H\) and \(J\), and the condition \(T_0=0\) gives
\(\nabla=\nabla^g\). The orthogonal parallel distributions \(H\) and
\(\mathbb R\xi\) give a local Riemannian product \(N^{2n}\times I\). On
each horizontal leaf \(N\), the induced Levi-Civita connection preserves
\(J|_{TN}\), so \(N\) is K\"ahler.

Since \(\nabla\) preserves \(P\), its curvature is \(\mathfrak g\)-valued;
hence \(R_0(X,e)=0\), \(R_0(X,Y)e=0\), and the horizontal block \(B\),
which represents the curvature of \(N\) in adapted frames, takes values
in \(\mathfrak u(n)\). It satisfies \eqref{eq:B-first-bianchi}, so
Lemma~\ref{lem:unitary-kahler-curvature} gives \(B=\kappa K\). The
parameter \(\kappa\) is constant because \(B\) is a block of the
characteristic tensor \(R_0\), and by \eqref{eq:kahler-normalization} it
is the holomorphic sectional curvature of \(N\).

Conversely, the Levi-Civita connection of such a product preserves the
induced \(G\)-structure, so its torsion and inner torsion vanish, and its
curvature has the representative \(\kappa K\) in the horizontal block of
every adapted frame. Hence the product is infinitesimally homogeneous.
\end{proof}

\subsubsection{The non-compatible case}
\label{subsubsec:non-compatible}

Assume now \(\mathfrak I_0\neq0\). By \eqref{eq:delta-alpha-a}, either
\(a=\alpha\), or \(a\neq\alpha\) and \(\delta=0\). We treat the two
branches separately.

\medskip\noindent\emph{The branch \(a=\alpha\).}
\label{subsubsec:branch-a-alpha}
Assume \(a=\alpha\). Then
\(\Pend(X)=aX+\beta J_0X\), \(\Qend(X)=\gamma X+\delta J_0X\), and the mixed
curvature coefficients become
\begin{equation}
\label{eq:c-branch-i}
\begin{aligned}
c_1&=a(b-a)+\beta^2,
&
c_2&=\beta(b-2a),\\
c_3&=(a-b)\gamma+\beta\delta,
&
c_4&=(a-b)\delta-\beta\gamma,
\end{aligned}
\quad
c_0=2c_4.
\end{equation}

Since \(c_0-2c_4=4\delta(\alpha-a)=0\), the mixed first Bianchi
identity is automatic. To impose the horizontal identity, let
\[
B_\lambda(X,Y)
=\Pend(X)\otimes\Qend(Y)^\flat-\Pend(Y)\otimes\Qend(X)^\flat
-2a\delta\,\omega_0(X,Y)\operatorname{Id}_{H_0}.
\]
By \eqref{eq:B-congruence-torsion-free}, \(B-B_\lambda\) takes
values in \(\mathfrak u(n)\). Regrouping the cyclic sum by the
coefficients of \(\Pend(X)\), \(\Pend(Y)\), and \(\Pend(Z)\), we obtain
\[
\mathfrak S_{X,Y,Z} B_\lambda(X,Y)Z
=\mathfrak S_{X,Y,Z}
\bigl(h_0(\Qend(Y),Z)-h_0(\Qend(Z),Y)\bigr)\Pend(X)
-2a\delta\,\mathfrak S_{X,Y,Z}\omega_0(X,Y)Z.
\]
Since \(h_0(\Qend(Y),Z)-h_0(\Qend(Z),Y)=2\delta\,\omega_0(Y,Z)\),
reindexing the last cyclic sum and using \(\Pend(X)-aX=\beta J_0X\)
gives
\[
\mathfrak S_{X,Y,Z} B_\lambda(X,Y)Z
=2\delta\,\mathfrak S_{X,Y,Z}\,
\omega_0(Y,Z)\bigl(\Pend(X)-aX\bigr)
=2\beta\delta\,\mathfrak S_{X,Y,Z}\,\omega_0(Y,Z)J_0X.
\]
Thus \(B^0(X,Y):=B_\lambda(X,Y)-2\beta\delta\,\omega_0(X,Y)J_0\)
satisfies \eqref{eq:B-first-bianchi} and, since \(J_0\in\mathfrak u(n)\),
has the same class as \(B_\lambda\) modulo \(\mathfrak u(n)\). Since both
\(B\) and \(B^0\) satisfy \eqref{eq:B-first-bianchi}, the difference
\(B-B^0\) is a \(\operatorname{U}(n)\)-equivariant map with values in
\(\mathfrak u(n)\) satisfying the first Bianchi identity, and
Lemma~\ref{lem:unitary-kahler-curvature} gives \(B-B^0=\kappa K\) for
some \(\kappa\in\mathbb R\). Hence
\begin{equation}
\label{eq:B-branch-i}
B(X,Y)
=B_\lambda(X,Y)-2\beta\delta\,\omega_0(X,Y)J_0+\kappa K_{X,Y}.
\end{equation}

By Lemmas~\ref{lem:bianchi2-kappa} and~\ref{lem:lie-realization}
below, the second Bianchi identity imposes no further restriction in
this branch. Thus \(\kappa\) remains free.

\medskip\noindent\emph{The branch \(a\neq\alpha\).}
\label{subsubsec:branch-a-not-alpha}
Assume \(a\neq\alpha\). By \eqref{eq:delta-alpha-a}, we have
\(\delta=0\), so \(\Qend(X)=\gamma X\). The curvature coefficients are
\begin{equation}
\label{eq:c-branch-ii}
\begin{aligned}
c_1&=-\alpha^2+b\alpha+\beta^2,
&
c_2&=\beta(b-2\alpha),\\
c_3&=(2a-\alpha-b)\gamma,
&
c_4&=-\beta\gamma,
\end{aligned}
\qquad c_0=2c_4.
\end{equation}

By \eqref{eq:B-congruence-torsion-free}, \(B-B_\lambda\) takes values in
\(\mathfrak u(n)\), where
\[
B_\lambda(X,Y)=
\gamma\bigl(\Pend(X)\otimes Y^\flat-\Pend(Y)\otimes X^\flat\bigr).
\]
The symmetry of \(h_0\) gives \(\mathfrak S_{X,Y,Z}B_\lambda(X,Y)Z=0\).
Thus \(B\) satisfies the horizontal first Bianchi identity precisely
when the \(\operatorname{U}(n)\)-equivariant difference \(B-B_\lambda\)
does. By Lemma~\ref{lem:unitary-kahler-curvature}, this is equivalent to
\(B=B_\lambda+\kappa K\) for some \(\kappa\in\mathbb R\).

Lemmas~\ref{lem:bianchi2-kappa} and~\ref{lem:lie-realization} below show
that the second Bianchi identity is equivalent to \(\kappa(\alpha-a)=0\).
Since \(a\neq\alpha\), we conclude that
\begin{equation}
\label{eq:B-branch-ii}
B(X,Y)=
\gamma\bigl(\Pend(X)\otimes Y^\flat-\Pend(Y)\otimes X^\flat\bigr).
\end{equation}

The second Bianchi identity is handled by two lemmas.
Lemma~\ref{lem:bianchi2-kappa} computes how its left-hand side changes
when \(\kappa K\) is added to the horizontal block: only the triples
\((e,X,Y)\) are affected, by \(2\kappa(\alpha-a)K_{X,Y}\oplus0\).
Lemma~\ref{lem:lie-realization} shows, by realizing them through
left-invariant connections on Lie groups, that the data with
\(\kappa=0\) satisfy it. Together they give the conclusions stated in
both branches.

For a fixed lift \(\lambda\), write
\[
\Sigma_\lambda(R)(u,v,w)
:=\mathfrak S_{u,v,w}(D_{\lambda(u)}R)(v,w)
\]
for the left-hand side of \eqref{eq:second-bianchi}.

\begin{lema}\label{lem:bianchi2-kappa}
Let \(\lambda\) be a lift in the normal form
\eqref{eq:inner-torsion-lift}, and let \(R_0^0\) be a fixed curvature
candidate. Extend \(K\) to \(V\) by setting
\(\widehat K(X,Y)=K_{X,Y}\oplus0\) and \(\widehat K(X,e)=0\) for
\(X,Y\in H_0\). Set \(R_0^\kappa=R_0^0+\kappa\widehat K\). The difference \(\Sigma_\lambda(R_0^\kappa)
-\Sigma_\lambda(R_0^0)\) vanishes on triples of horizontal
vectors and on triples with two vertical entries. For
\(X,Y\in H_0\), it satisfies
\[
\bigl(\Sigma_\lambda(R_0^\kappa)
-\Sigma_\lambda(R_0^0)\bigr)(e,X,Y)
=2\kappa(\alpha-a)(K_{X,Y}\oplus0).
\]
\end{lema}

\begin{proof}
By linearity, the difference is \(\kappa\,\Sigma_\lambda(\widehat K)\).
Since \(\Sigma_\lambda(\widehat K)\) is alternating, it suffices to
consider triples \((X,Y,Z)\) and \((e,X,Y)\) with \(X,Y,Z\in H_0\).

\emph{Triples \((X,Y,Z)\).} Since \(\lambda(X)Y\in\mathbb Re\) and
\(\widehat K(e,\cdot)=0\), we have
\[
\Sigma_\lambda(\widehat K)(X,Y,Z)
=\mathfrak S_{X,Y,Z}[\lambda(X),K_{Y,Z}\oplus0].
\]
Its two off-diagonal
blocks are \(-\mathfrak S\,K_{Y,Z}\Pend(X)\) and the covector
\(W\mapsto-h_0\bigl(\mathfrak S\,K_{Y,Z}\Qend(X),W\bigr)\). Both vanish,
because \(\Pend,\Qend\in\mathbb R\operatorname{Id}_{H_0}+\mathbb RJ_0\),
each \(K_{Y,Z}\) commutes with \(J_0\), and
\(\mathfrak S_{X,Y,Z}K_{Y,Z}X=0\).

\emph{Triples \((e,X,Y)\).} Since
\(\lambda(e)=a\operatorname{Id}_{H_0}\oplus b\) commutes with
\(K_{X,Y}\oplus0\) and \(\lambda(e)X=aX\), we have
\((D_{\lambda(e)}\widehat K)(X,Y)=-2aK_{X,Y}\oplus0\). Since
\(\lambda(X)e=\Pend(X)\), \(\lambda(Y)X\in\mathbb Re\) and
\(\widehat K(\cdot,e)=0\), the two remaining terms of the cyclic sum are
\[
-\widehat K\bigl(Y,\Pend(X)\bigr)-\widehat K\bigl(\Pend(Y),X\bigr)
=\bigl(K_{\Pend(X),Y}+K_{X,\Pend(Y)}\bigr)\oplus0 .
\]
As \(\Pend=\alpha\operatorname{Id}_{H_0}+\beta J_0\) and
\(K_{J_0X,Y}=-K_{X,J_0Y}\), this equals \(2\alpha K_{X,Y}\oplus0\).
Adding up, \(\Sigma_\lambda(\widehat K)(e,X,Y)=2(\alpha-a)K_{X,Y}\oplus0\).
\end{proof}

\begin{lema}[Realization for \(\kappa=0\)]\label{lem:lie-realization}
Let \(\mathfrak I_0\) be given in the normal form
\eqref{eq:inner-torsion-lift} with \(\delta(\alpha-a)=0\), and let
\(R_0\) be the curvature candidate of branch \textup{(i)} if
\(\alpha=a\), or of branch \textup{(ii)} if \(\alpha\neq a\), with
\(\kappa=0\). Then \((R_0,\mathfrak I_0)\) are the curvature and inner
torsion characteristic tensors of a torsion-free infinitesimally
homogeneous affine manifold with \(G\)-structure. In particular, they
satisfy \eqref{eq:second-bianchi}.
\end{lema}

\begin{proof}
Define a skew-symmetric bracket on \(V\) by
\[
[X,Y]=2\delta\,\omega_0(X,Y)\,e,\quad [X,e]=(\alpha-a)X,
\qquad X,Y\in H_0.
\]
The Jacobi identity is trivial on triples with two vertical entries,
while
\[
\mathfrak S_{X,Y,Z}[[X,Y],Z]=-2\delta(\alpha-a)\,
\mathfrak S_{X,Y,Z}\,\omega_0(X,Y)Z,
\qquad
\mathfrak S_{X,Y,e}[[X,Y],e]=-4\delta(\alpha-a)\,\omega_0(X,Y)\,e,
\]
and both vanish. Let \(L\) be a Lie group with Lie algebra
\((V,[\cdot,\cdot])\), let \(s\) be the frame of left-invariant vector
fields determined by \(V\), and let \(P=\{s(x)g:x\in L,\ g\in G\}\).
Put \(\lambda'(X)=\lambda(X)\) and
\(\lambda'(e)=\lambda(e)+\beta\,(J_0\oplus0)\), and let
\(\nabla\) be the left-invariant connection with
\(\nabla_{s u}(s\,w)=s\,\lambda'(u)w\). A direct check shows
\(\lambda'(u)v-\lambda'(v)u=[u,v]\), so \(\nabla\) is torsion-free, and
its curvature is represented in \(s\) by
\(R_0'(u,v)=[\lambda'(u),\lambda'(v)]-\lambda'([u,v])\).

The map \(\lambda'\) differs from \(\lambda\) by a \(\mathfrak g\)-valued
map and, like \(\lambda\), is \(G\)-equivariant on the nose
(Remark~\ref{rem:inner-torsion-parameters}), since \(J_0\) commutes with
\(\operatorname{U}(n)\). Hence the bracket and \(R_0'\) are
\(G\)-equivariant, so \(T\), \(R\) and the inner torsion have the same
constant representatives in all frames of \(P\): the triple
\((L,\nabla,P)\) is infinitesimally homogeneous, with inner torsion
\(\mathfrak I_0\) and curvature \(R_0'\). The coefficients
\(c_0,\dots,c_4\) of \(R_0'\) are determined by \(\mathfrak I_0\)
through \eqref{eq:torsion-free-compatibility}, and its horizontal block
is that of \([\lambda(X),\lambda(Y)]\) minus
\(2\delta\,\omega_0(X,Y)\lambda'(e)\), namely
\[
\Pend(X)\otimes\Qend(Y)^\flat-\Pend(Y)\otimes\Qend(X)^\flat
-2\delta\,\omega_0(X,Y)\bigl(a\operatorname{Id}_{H_0}+\beta J_0\bigr).
\]
This is \eqref{eq:B-branch-i} with \(\kappa=0\) if \(\alpha=a\), and
\eqref{eq:B-branch-ii} if \(\delta=0\). Hence \(R_0'=R_0\).
\end{proof}

For \(\alpha=a\) and \(\delta\neq0\), the Lie algebra
\((V,[\cdot,\cdot])\) is the Heisenberg algebra, and for \(\delta=0\),
\(\alpha\neq a\), it is the Lie algebra of the solvable group
\(\mathbb R\ltimes\mathbb R^{2n}\) underlying real hyperbolic space;
compare the models of Subsection~\ref{subsec:underlying}.

\begin{teo}[Classification of the torsion-free non-compatible case]
\label{thm:torsion-free-non-compatible}
Let \(G=\operatorname{U}(n)\times\{1\}\), let
\(\mathfrak I_0\neq0\) be a \(G\)-equivariant inner torsion with normal
form parameters \((a,b,\alpha,\beta,\gamma,\delta)\), and let \(R_0\) be
a \(G\)-equivariant curvature candidate as in
Proposition~\ref{prop:curvature-candidates-Un}. Then
\((R_0,\mathfrak I_0)\) are the characteristic tensors of a torsion-free
infinitesimally homogeneous affine manifold with \(G\)-structure if and
only if \(\delta(\alpha-a)=0\) and, according to whether \(\alpha=a\) or
\(\alpha\neq a\), the following holds:
\begin{enumerate}
\item[(i)] \(\alpha=a\); the coefficients \(c_0,\dots,c_4\) are given
by \eqref{eq:c-branch-i}, and there exists \(\kappa\in\mathbb R\) such
that \(B\) is given by \eqref{eq:B-branch-i}. Here
\((a,b,\beta,\gamma,\delta,\kappa)\) are arbitrary, subject only to
\((a,b,\beta,\gamma,\delta)\neq0\);
\item[(ii)] \(\alpha\neq a\); then \(\delta=0\), the coefficients
\(c_0,\dots,c_4\) are given by \eqref{eq:c-branch-ii}, and \(B\) is given
by \eqref{eq:B-branch-ii}. Here \((a,b,\alpha,\beta,\gamma)\) are
arbitrary with \(a\neq\alpha\).
\end{enumerate}
\end{teo}

\begin{proof}
We showed above that the compatibility relation and the first Bianchi
identity give \eqref{eq:delta-alpha-a} and the coefficients
\eqref{eq:c-branch-i}, respectively \eqref{eq:c-branch-ii}. Before imposing the second Bianchi identity, the horizontal block
has the form \(B=B^0+\kappa K\) in both branches, where \(B^0\)
is the right-hand side of \eqref{eq:B-branch-i} with \(\kappa=0\)
in branch \textup{(i)}, and the right-hand side of
\eqref{eq:B-branch-ii} in branch \textup{(ii)}.
By Lemma~\ref{lem:lie-realization}, the data of either branch with
\(\kappa=0\) satisfy \eqref{eq:second-bianchi}. By
Lemma~\ref{lem:bianchi2-kappa}, for K\"ahler parameter \(\kappa\) the
left-hand side of \eqref{eq:second-bianchi} differs from its value at
\(\kappa=0\) by \(2\kappa(\alpha-a)K_{X,Y}\oplus0\) on the triples
\((e,X,Y)\), and by zero on all other triples. In branch (i) this
difference vanishes, so \eqref{eq:second-bianchi} holds for every
\(\kappa\); in branch (ii) it vanishes if and only if \(\kappa=0\).
Theorem~\ref{thm:algebraic-characterization} completes the proof.
\end{proof}

\begin{remark}
\label{rem:compatible-inside-branch-i}
Setting \(a=b=\beta=\gamma=\delta=0\) in branch (i) gives
\(c_0=\dots=c_4=0\) and \(B=\kappa K\), which is precisely the compatible
family of Theorem~\ref{thm:compatible-torsion-free}. Thus the formulas
of branch (i) also cover the compatible case.
\end{remark}

%%%%%%%%%%%%%%%%%%%%%%%%%
%%%%%%%%%%%%%%%%%%%%%%%%%

\subsection{Invariant torsion deformations}
\label{subsec:unitary-torsion-deformations}

We now pass from the torsion-free classification to the general case.
As explained in Subsection~\ref{subsec:deformations}, every
infinitesimally homogeneous connection is obtained from its torsion-free
symmetrization by adding one half of the tensor induced by an
element of \(\operatorname{Hom}_G(\Lambda^2V,V)\).
Proposition~\ref{prop:invariant-torsion-Un} shows that every such tensor
is uniquely determined by three parameters
\(\tau_1,\tau_2,\tau_3\in\mathbb R\) and has the form
\begin{equation}
\label{eq:general-torsion-deformation-Un}
t_0(X,Y)=\tau_3\,\omega_0(X,Y)e,
\quad
t_0(X,e)=\tau_1X+\tau_2J_0X,
\quad
X,Y\in H_0.
\end{equation}

Let \((M,\nabla^s,P)\) be any torsion-free infinitesimally homogeneous
manifold with \(G\)-structure, and let \(t\) be the tensor induced by
\(t_0\). Then
\(
\nabla=\nabla^s+\frac12t
\) defines a connection on \(M\) such that \((M,\nabla, P)\) is
infinitesimally homogeneous, with torsion characteristic tensor
\(T_0=t_0\). We now describe its inner torsion and curvature
characteristic tensors using
\eqref{eq:general-torsion-transformation}--\eqref{eq:general-curvature-deformation}
with \(A_0=t_0\). We write \(t_{0,u}:=t_0(u,\cdot)\).

\subsubsection{Transformation of the inner torsion}
\label{subsubsec:unitary-inner-deformation}

Let the inner torsion of the torsion-free base be represented by the
lift \(\lambda^s\) in normal form \eqref{eq:inner-torsion-lift}, with
parameters \(a^s,b^s,\alpha^s,\beta^s,\gamma^s,\delta^s\).
The endomorphisms \(t_{0,X}\) and \(t_{0,e}\) are
\[
t_{0,X}
=
\begin{pmatrix}
0&\tau_1X+\tau_2J_0X\\[1mm]
(\tau_3J_0X)^\flat&0
\end{pmatrix},
\quad
t_{0,e}
=
\begin{pmatrix}
-(\tau_1\operatorname{Id}_{H_0}+\tau_2J_0)&0\\
0&0
\end{pmatrix}.
\]
Hence, by \eqref{eq:general-inner-transformation},
\(\widehat\lambda(u):=\lambda^s(u)+\frac12t_{0,u}\) is a lift of the
inner torsion of \(\nabla\); this sum comes from the deformation of the
connection. It is not yet in normal form: the horizontal block of
\(\widehat\lambda(e)\) is
\(\bigl(a^s-\frac{\tau_1}{2}\bigr)\operatorname{Id}_{H_0}-\frac{\tau_2}{2}J_0\).
Since \(J_0\in\mathfrak u(n)\), the lift
\(\lambda:=\widehat\lambda+\theta\), where \(\theta\) is the
\(\mathfrak g\)-valued linear map with \(\theta(X)=0\) and
\(\theta(e)=\frac{\tau_2}{2}J_0\oplus0\), represents the same inner
torsion; this second sum only changes the representative. The lift
\(\lambda\) is in the normal form \eqref{eq:inner-torsion-lift}, and
comparing parameters we obtain that the six parameters of the inner
torsion transform according to
\begin{equation}
\label{eq:inner-parameter-transformation}
\begin{aligned}
a&=a^s-\frac{\tau_1}{2},
&
b&=b^s,
&
\alpha&=\alpha^s+\frac{\tau_1}{2},\\
\beta&=\beta^s+\frac{\tau_2}{2},
&
\gamma&=\gamma^s,
&
\delta&=\delta^s+\frac{\tau_3}{2}.
\end{aligned}
\end{equation}
Conversely, \(\delta^s=\delta-\frac{\tau_3}{2}\) and
\(\alpha^s-a^s=\alpha-a-\tau_1\); these two combinations will reappear
as the trans-Sasakian invariants \(\alpha_T\) and \(\beta_T\) in
Theorem~\ref{thm:structure}.

The horizontal component \(\Pend_0(X):=t_0(X,e)=\tau_1X+\tau_2J_0X\)
of the torsion enters all the curvature formulas below, and in terms of
it
\begin{equation}
\label{eq:PQ-deformed}
\Pend(X)
=\left(\alpha^s+\frac{\tau_1}{2}\right)X
+\left(\beta^s+\frac{\tau_2}{2}\right)J_0X
=\Pend^s(X)+\frac12\Pend_0(X),
\end{equation}
and
\begin{equation}
\label{eq:Q-deformed}
\Qend(X)
=\gamma^sX+\left(\delta^s+\frac{\tau_3}{2}\right)J_0X
=\Qend^s(X)+\frac{\tau_3}{2}J_0X.
\end{equation}

\subsubsection{General curvature deformation formulas}
\label{subsubsec:unitary-curvature-deformation}

Let \(c_0^s,c_1^s,c_2^s,c_3^s,c_4^s\) and
\(B^s:\Lambda^2H_0\to\mathfrak{gl}(H_0)\) denote the curvature
coefficients and horizontal curvature block of the torsion-free base.
Thus
\[
R_0^s(X,e)
=
\begin{pmatrix}
0&c_1^sX+c_2^sJ_0X\\[1mm]
(c_3^sX+c_4^sJ_0X)^\flat&0
\end{pmatrix},
\quad
R_0^s(X,Y)
=
\begin{pmatrix}
B^s(X,Y)&0\\
0&c_0^s\,\omega_0(X,Y)
\end{pmatrix}.
\]
Since the base connection is torsion-free,
\eqref{eq:general-curvature-deformation} gives
\begin{equation}
\label{eq:curvature-deformation-Un}
R_0(u,v)w
=R_0^s(u,v)w
+\frac12\bigl(D_{\lambda^s(u)}t_0\bigr)(v,w)
-\frac12\bigl(D_{\lambda^s(v)}t_0\bigr)(u,w)
+\frac14[t_{0,u},t_{0,v}]w.
\end{equation}
We now evaluate this formula on the arguments determined by the
decomposition \(V=H_0\oplus\mathbb Re\).

\begin{itemize}
\item[(a)]
For the component \(R_0(X,e)e\) we obtain
\[
R_0(X,e)e=R_0^s(X,e)e+\frac{b^s}{2}\Pend_0(X)+\frac14\Pend_0^2(X).
\]
Comparing the coefficients it follows that
\begin{equation}
\label{eq:c1c2-general-deformation}
c_1=c_1^s+\frac12b^s\tau_1+\frac14(\tau_1^2-\tau_2^2),
\quad
c_2=c_2^s+\frac12b^s\tau_2+\frac12\tau_1\tau_2.
\end{equation}

\item[(b)] For the component \(R_0(X,e)Y\) we obtain
\[
\begin{aligned}
R_0(X,e)Y-R_0^s(X,e)Y
={}&
\Bigl[
-\frac12h_0(\Qend^s(X),\Pend_0(Y))
-\frac{\tau_3}{2}\omega_0(\Pend^s(X),Y)\\
&\quad
+\frac{(2a^s-b^s)\tau_3}{2}\omega_0(X,Y)
-\frac{\tau_3}{4}\omega_0(X,\Pend_0(Y))
\Bigr]e.
\end{aligned}
\]
Therefore, by comparing the coefficients,
\begin{equation}
\label{eq:c3c4-general-deformation}
\begin{aligned}
c_3
={}&
c_3^s-\frac12\gamma^s\tau_1-\frac12\delta^s\tau_2+\frac12\beta^s\tau_3-\frac14\tau_2\tau_3,\\
c_4
={}&
c_4^s-\frac12\delta^s\tau_1+\frac12\gamma^s\tau_2
+\left(a^s-\frac{\alpha^s+b^s}{2}\right)\tau_3-\frac14\tau_1\tau_3.
\end{aligned}
\end{equation}

\item[(c)] For the component \(R_0(X,Y)e\) we have
\[
R_0(X,Y)e-R_0^s(X,Y)e
=\Bigl[\delta^s\tau_1-\gamma^s\tau_2+\alpha^s\tau_3+\frac12\tau_1\tau_3\Bigr]\omega_0(X,Y)e.
\]
Therefore
\begin{equation}
\label{eq:c0-general-deformation}
c_0=c_0^s+\delta^s\tau_1-\gamma^s\tau_2+\alpha^s\tau_3+\frac12\tau_1\tau_3.
\end{equation}

\item[(d)] For the horizontal component \(R_0(X,Y)Z\) we obtain
\begin{equation}
\label{eq:B-general-deformation}
\begin{aligned}
B(X,Y)Z
={}&
B^s(X,Y)Z
+\frac{\tau_3}{2}\Bigl(\omega_0(Y,Z)\Pend(X)-\omega_0(X,Z)\Pend(Y)\Bigr)
+\delta^s\omega_0(X,Y)\Pend_0(Z)\\
&-\frac12h_0(\Qend^s(X),Z)\Pend_0(Y)+\frac12h_0(\Qend^s(Y),Z)\Pend_0(X).
\end{aligned}
\end{equation}
\end{itemize}

By the symmetrization--deformation correspondence, the general
classification is therefore obtained by specializing these
transformation formulas,
\eqref{eq:inner-parameter-transformation},
\eqref{eq:c1c2-general-deformation}, \eqref{eq:c3c4-general-deformation},
\eqref{eq:c0-general-deformation} and \eqref{eq:B-general-deformation},
to the torsion-free families classified above.

\subsubsection{Deformation of the compatible family}
\label{subsubsec:deformation-compatible-family}

For the compatible torsion-free family,
\(a^s=b^s=\alpha^s=\beta^s=\gamma^s=\delta^s=0\),
\(c_0^s=c_1^s=c_2^s=c_3^s=c_4^s=0\), and
\(B^s(X,Y)=\kappa^sK_{X,Y}\), where \(\kappa^s\in\mathbb R\). Therefore,
the previous discussion immediately gives the following family.

\begin{cor}
\label{cor:deformation-compatible-family}
A general invariant torsion deformation of the compatible torsion-free
family has inner torsion parameters
\[
a=-\frac{\tau_1}{2},\quad b=0, \quad
\alpha=\frac{\tau_1}{2}, \quad\beta=\frac{\tau_2}{2},\quad
\gamma=0,\quad \delta=\frac{\tau_3}{2}.
\]
Its curvature coefficients are
\[
c_1=\frac14(\tau_1^2-\tau_2^2),
\quad
c_2=\frac12\tau_1\tau_2,
\quad
c_3=-\frac14\tau_2\tau_3,
\quad
c_4=-\frac14\tau_1\tau_3,
\quad
c_0=\frac12\tau_1\tau_3.
\]
The horizontal curvature block is
\begin{equation}
\label{eq:B-deformed-compatible}
B(X,Y)Z
=\kappa^sK_{X,Y}Z
+\frac{\tau_3}{4}\Bigl(\omega_0(Y,Z)\Pend_0(X)-\omega_0(X,Z)\Pend_0(Y)\Bigr).
\end{equation}
\end{cor}

\subsubsection{Deformation of the non-compatible family}
\label{subsubsec:deformation-non-compatible}

\begin{cor}
\label{cor:deformation-non-compatible}
Let the torsion-free base have nonzero inner torsion, and let
\(\tau_1,\tau_2,\tau_3\in\mathbb R\) be the parameters of an invariant
torsion deformation. The deformed inner torsion parameters are given
by \eqref{eq:inner-parameter-transformation}. Its curvature
coefficients \(c_0,\dots,c_4\) are given by
\eqref{eq:c1c2-general-deformation},
\eqref{eq:c3c4-general-deformation}, and
\eqref{eq:c0-general-deformation}, with the base coefficients
\(c_0^s,\dots,c_4^s\) specified in each branch below.
In both branches, \(\Pend\) and \(\Qend\) are given by
\eqref{eq:PQ-deformed} and \eqref{eq:Q-deformed}.

\begin{enumerate}
\item[(i)] If \(a^s=\alpha^s\), the base coefficients are given by
\eqref{eq:c-branch-i}, evaluated at the base parameters, and
\begin{equation}
\label{eq:B-deformed-branch-i}
\begin{aligned}
B(X,Y)Z
={}&\Pend(X)h_0(\Qend(Y),Z)-\Pend(Y)h_0(\Qend(X),Z)\\
&-2a\delta^s\,\omega_0(X,Y)Z
+\delta^s(\tau_2-2\beta^s)\omega_0(X,Y)J_0Z
+\kappa^sK_{X,Y}Z,
\end{aligned}
\end{equation}
where \(\kappa^s\in\mathbb R\) is the free curvature parameter of the
base.

\item[(ii)] If \(a^s\neq\alpha^s\), then \(\delta^s=\kappa^s=0\), the
base coefficients are given by \eqref{eq:c-branch-ii}, evaluated at the
base parameters, and
\begin{equation}
\label{eq:B-deformed-branch-ii}
B(X,Y)Z
=\Pend(X)h_0(\Qend(Y),Z)-\Pend(Y)h_0(\Qend(X),Z).
\end{equation}
\end{enumerate}
\end{cor}

\begin{proof}
The inner torsion parameters and the coefficients \(c_0,\dots,c_4\)
follow by substitution in the general deformation formulas. To compute
the horizontal block, set
\[
\mathcal A(E,F)(X,Y)=E(X)\otimes F(Y)^\flat-E(Y)\otimes F(X)^\flat
\]
for endomorphisms \(E,F\) of \(H_0\). By \eqref{eq:PQ-deformed} and
\eqref{eq:Q-deformed},
\[
\mathcal A(\Pend,\Qend)-\mathcal A(\Pend^s,\Qend^s)
=\frac12\mathcal A(\Pend_0,\Qend^s)
+\frac{\tau_3}{2}\mathcal A(\Pend,J_0).
\]
Consequently, \eqref{eq:B-general-deformation} reads
\[
B(X,Y)
=B^s(X,Y)-\mathcal A(\Pend^s,\Qend^s)(X,Y)
+\mathcal A(\Pend,\Qend)(X,Y)
+\delta^s\omega_0(X,Y)\Pend_0.
\]
In branch \textup{(i)}, substituting \eqref{eq:B-branch-i} for the base
block and using \(\Pend_0=\tau_1\operatorname{Id}_{H_0}+\tau_2J_0\) and
\(a=a^s-\tau_1/2\) gives \eqref{eq:B-deformed-branch-i}. In branch
\textup{(ii)}, \eqref{eq:B-branch-ii} gives
\(B^s=\mathcal A(\Pend^s,\Qend^s)\), and \(\delta^s=0\) gives
\eqref{eq:B-deformed-branch-ii}.
\end{proof}

\subsubsection{Exhaustiveness of the deformation description}
\label{subsubsec:unitary-exhaustiveness}

\begin{teo}[Classification through symmetrization]
\label{thm:unitary-general-classification}
Let \(G=\operatorname{U}(n)\times\{1\}\). The assignments
\[
\nabla\longmapsto\bigl(\nabla-\tfrac12T,\;T_0\bigr),
\qquad
(\nabla^s,t_0)\longmapsto\nabla^s+\tfrac12t
\]
are mutually inverse maps between the infinitesimally homogeneous
connections \(\nabla\) on \((M,P)\) and the pairs \((\nabla^s,t_0)\),
where \(\nabla^s\) is a torsion-free infinitesimally homogeneous
connection on \((M,P)\) and \(t_0\in\operatorname{Hom}_G(\Lambda^2V,V)\)
is given by \eqref{eq:general-torsion-deformation-Un}. The connection
\(\nabla^s\) belongs to exactly one of the following families:
\begin{enumerate}
\item the compatible family of Theorem~\ref{thm:compatible-torsion-free};
\item the branch \(a^s=\alpha^s\), with \(\mathfrak I_0^s\neq0\), of
Theorem~\ref{thm:torsion-free-non-compatible};
\item the branch \(a^s\neq\alpha^s\) of
Theorem~\ref{thm:torsion-free-non-compatible}.
\end{enumerate}
The characteristic tensors of \(\nabla\) are those of
Corollary~\ref{cor:deformation-compatible-family} in case~(1), and of
Corollary~\ref{cor:deformation-non-compatible} in cases~(2) and~(3).
\end{teo}

\begin{proof}
By Subsection~\ref{subsec:deformations}, applied with \(A_0=-T_0\) and
with \(A_0=t_0\), both assignments take values in the stated sets, and
by Proposition~\ref{prop:invariant-torsion-Un}, \(T_0\) has the form
\eqref{eq:general-torsion-deformation-Un}. The two compositions are the
identity: \(\bigl(\nabla-\frac12T\bigr)+\frac12T=\nabla\), and, since the
torsion of \(\nabla^s+\frac12t\) is \(t\), subtracting half of it gives
back \(\nabla^s\) and its torsion characteristic tensor is \(t_0\). The
three families are disjoint: the first has \(\mathfrak I_0^s=0\), and
the other two are distinguished by whether \(a^s=\alpha^s\). By
Theorems~\ref{thm:compatible-torsion-free}
and~\ref{thm:torsion-free-non-compatible}, they exhaust the
torsion-free case. The last assertion is the content of
Corollaries~\ref{cor:deformation-compatible-family}
and~\ref{cor:deformation-non-compatible}.
\end{proof}

The classification can be stated directly in terms of the parameters of
the characteristic tensors of \(\nabla\).

\begin{cor}
\label{cor:parameter-classification}
Let \(\tau_1,\tau_2,\tau_3\) and \(a,b,\alpha,\beta,\gamma,\delta\) be
real numbers, defining \(G\)-equivariant tensors \(T_0\) and
\(\mathfrak I_0\) through \eqref{eq:torsion-normal-form-Un} and
\eqref{eq:inner-torsion-lift}. Then \((T_0,\mathfrak I_0)\) are the
torsion and inner torsion characteristic tensors of an infinitesimally
homogeneous affine manifold with structure group
\(\operatorname{U}(n)\times\{1\}\) if and only if
\begin{equation}
\label{eq:parameter-relation}
(2\delta-\tau_3)(\alpha-a-\tau_1)=0.
\end{equation}
In that case, the curvature characteristic tensor is uniquely
determined once, in addition, the K\"ahler parameter \(\kappa\) is
specified; this parameter is arbitrary when \(\alpha-a=\tau_1\) and must
vanish otherwise. The curvature tensor is given by
Corollaries~\ref{cor:deformation-compatible-family}
and~\ref{cor:deformation-non-compatible} (see also
Remark~\ref{rem:compatible-inside-branch-i}) with
\[
a^s=a+\frac{\tau_1}{2},\quad
b^s=b,\quad
\alpha^s=\alpha-\frac{\tau_1}{2},\quad
\beta^s=\beta-\frac{\tau_2}{2},\quad
\gamma^s=\gamma,\quad
\delta^s=\delta-\frac{\tau_3}{2},\quad
\kappa^s=\kappa.
\]
\end{cor}

\begin{proof}
Inverting \eqref{eq:inner-parameter-transformation}, the inner torsion
parameters of the symmetrized connection are given by the displayed
formulas, and
\[
\delta^s(\alpha^s-a^s)=\tfrac12(2\delta-\tau_3)(\alpha-a-\tau_1).
\]
If the data come from an infinitesimally homogeneous manifold, the
symmetrized connection satisfies \eqref{eq:delta-alpha-a}, which gives
\eqref{eq:parameter-relation}; moreover \(\kappa^s=0\) unless
\(\alpha^s=a^s\). Conversely, if \eqref{eq:parameter-relation} holds, the
parameters \((a^s,\dots,\delta^s)\) satisfy \(\delta^s(\alpha^s-a^s)=0\).
By Theorems~\ref{thm:compatible-torsion-free}
and~\ref{thm:torsion-free-non-compatible} (see also
Remark~\ref{rem:compatible-inside-branch-i}), there is a torsion-free
infinitesimally homogeneous manifold with these inner torsion
parameters and with K\"ahler parameter \(\kappa^s\) (with \(\kappa^s=0\) if
\(\alpha^s\neq a^s\)). Deforming the resulting torsion-free connection
by the tensor induced by \(t_0=T_0\) gives an infinitesimally
homogeneous connection with the prescribed torsion and inner torsion
characteristic tensors. For each admissible \(\kappa=\kappa^s\), the
torsion-free classification determines \(R_0^s\) uniquely. The curvature
deformation formulas then determine \(R_0\) uniquely, as stated in
Corollaries~\ref{cor:deformation-compatible-family}
and~\ref{cor:deformation-non-compatible}.
\end{proof}

\subsection{The structure theorem and its geometric consequences}
\label{subsec:underlying}

The classification above is expressed in terms of the
characteristic tensors of the connection. The underlying
\(G\)-structure \((\phi,\xi,\eta,g)\), however, has invariants of its
own, independent of \(\nabla\); for instance, whether the horizontal
distribution \(H=\ker\eta\) is integrable, contact, or neither, which by
Frobenius is measured by \(d\eta|_{H\times H}\). We now identify the
almost contact metric manifolds underlying the classification. Our main
tool is the Levi-Civita connection \(\nabla^g\) of the metric \(g\)
determined by \(P\), which is again infinitesimally homogeneous by
Proposition~\ref{prop:local-homogeneity}(c).

Recall \cite{oubina,Blair} that an almost contact metric structure
\((\phi,\xi,\eta,g)\) is \emph{trans-Sasakian of type
\((\alpha_T,\beta_T)\)} if, for the Levi-Civita connection \(\nabla\)
of \(g\),
\[
(\nabla_X\phi)Y
=\alpha_T\bigl(g(X,Y)\xi-\eta(Y)X\bigr)
+\beta_T\bigl(g(\phi X,Y)\xi-\eta(Y)\phi X\bigr);
\]
in that case
\(\nabla_X\xi=-\alpha_T\phi X+\beta_T\bigl(X-\eta(X)\xi\bigr)\). Type
\((0,0)\) is cosymplectic (coK\"ahler), type \((\alpha_T,0)\) with
\(\alpha_T\neq0\) is \(\alpha_T\)-Sasakian (Sasakian if \(\alpha_T=1\)),
and type \((0,\beta_T)\) with \(\beta_T\neq0\) is \(\beta_T\)-Kenmotsu
\cite{Kenmotsu}.

\begin{prop}
\label{prop:levi-civita}
Let \((M,\nabla,P)\) be an infinitesimally homogeneous affine manifold
with structure group \(\operatorname{U}(n)\times\{1\}\), \(T_0=0\), and
inner torsion parameters \((a,b,\alpha,\beta,\gamma,\delta)\).
Then \(\nabla\) is the Levi-Civita connection of
\(g=h+\eta\otimes\eta\) if and only if
\[
a=b=0,
\quad
\gamma=-\alpha,
\quad
\delta=-\beta.
\]
In this case,
\[
\nabla_X\xi=\alpha X+\beta JX,
\quad X\in H,
\]
and the associated almost contact metric structure is trans-Sasakian
of type
\(
(\alpha_T,\beta_T)=(-\beta,\alpha).
\)
Moreover,
\(
\alpha_T\beta_T=0.
\)
\end{prop}
\begin{proof}
Since \(T_0=0\), the connection is Levi-Civita if and only if it is
metric. In an adapted orthonormal frame this is equivalent to the
connection matrices being skew-symmetric. Since
\(\Lambda^s(u)\in\lambda(u)+\mathfrak g\) and
\(\mathfrak g\subset\mathfrak{so}(2n+1)\), this means that \(\lambda(u)\)
is skew-symmetric for every \(u\in V\). For the normal form
\eqref{eq:inner-torsion-lift} this gives
\(
a=b=0,
\)
\(\gamma=-\alpha,
\) and
\(\delta=-\beta.
\)
Hence, by \eqref{eq:lift-convention},
\[
\nabla_X\xi=\alpha X+\beta JX,
\quad X\in H,
\quad
\nabla_\xi\xi=0.
\]
Moreover, \(\phi\) has constant components
\(\phi_0=J_0\oplus0\), which commutes with \(\mathfrak g\), so
\(\nabla_{pu}\phi\) is represented by \([\lambda(u),\phi_0]\). A direct
computation gives \([\lambda(e),\phi_0]=0\) and, for \(X,Y\in H_0\),
\[
[\lambda(X),\phi_0]\,e=\beta X-\alpha J_0X,
\quad
[\lambda(X),\phi_0]\,Y=\bigl(-\beta\,h_0(X,Y)+\alpha\,\omega_0(X,Y)\bigr)e,
\]
which is the trans-Sasakian identity with
\(
\alpha_T=-\beta\) and \(\beta_T=\alpha.
\)
Finally, the torsion-free relation
\(
\delta(\alpha-a)=0
\)
reduces to
\(
-\beta\alpha=0,
\)
hence
\(
\alpha_T\beta_T=0.
\)
\end{proof}

By
Lemmas~\ref{lem:unitary-facts} and~\ref{lem:parity}, every
\(G\)-equivariant bilinear map \(S_0:V\times V\to V\) has the form
\begin{equation}\label{eq:S0}
\begin{gathered}
S_0(X,Y)=\bigl(\sigma_1h_0(X,Y)+\sigma_2\,\omega_0(X,Y)\bigr)e,\qquad
S_0(X,e)=\sigma_3X+\sigma_4J_0X,\\
S_0(e,X)=\sigma_5X+\sigma_6J_0X,\qquad S_0(e,e)=\sigma_7\,e,
\end{gathered}
\end{equation}
for \(X,Y\in H_0\) and \(\sigma_1,\dots,\sigma_7\in\mathbb R\). In
particular, \(\dim\operatorname{Hom}_G(V\otimes V,V)=7\).

\medskip\noindent\textbf{The three models.} Let
\(\alpha_T,\beta_T,\kappa\in\mathbb R\) with \(\alpha_T\beta_T=0\).
\begin{enumerate}
\item[(C)] \(\alpha_T=\beta_T=0\): the product
\(N^{2n}(\kappa)\times\mathbb R\) of a complex space form of constant
holomorphic sectional curvature \(\kappa\) and a line, with its
cosymplectic structure.
\item[(S)] \(\alpha_T\neq0\): the \(\alpha_T\)-Sasakian space form of
constant \(\phi\)-sectional curvature \(c=\kappa-3\alpha_T^2\).

When \(\alpha_T>0\), it is obtained from a Sasakian space
form~\cite[Chapter~7]{Blair} by the homothety \(g\mapsto\alpha_T^{-2}g\),
\(\xi\mapsto\alpha_T\xi\), \(\eta\mapsto\alpha_T^{-1}\eta\). When
\(\alpha_T<0\), one further replaces \((\xi,\eta)\) by
\((-\xi,-\eta)\).
\item[(K)] \(\beta_T\neq0\) (and then \(\kappa:=0\)): the real hyperbolic
space \(\mathbb H^{2n+1}(-\beta_T^2)=\mathbb R\times_{e^{\beta_Tt}}\mathbb C^n\),
with metric \(dt^2+e^{2\beta_Tt}|dz|^2\), \(\xi=\partial_t\) and \(\phi\)
the complex structure of \(\mathbb C^n\). This is its \(\beta_T\)-Kenmotsu
structure~\cite{Kenmotsu}.
\end{enumerate}

\begin{prop}[Model realization]\label{prop:models}
\begin{enumerate}
\item[(i)] In each model the automorphisms of the almost contact metric
structure act transitively on the frames of \(P\). Hence
\((M,\nabla^g,P)\) is infinitesimally homogeneous, with inner torsion
parameters
\[
(a,b,\alpha,\beta,\gamma,\delta)=(0,0,\beta_T,-\alpha_T,-\beta_T,\alpha_T).
\]
Its curvature is
\begin{itemize}
\item that of Theorem~\ref{thm:compatible-torsion-free} in case (C);
\item that of branch (i), with K\"ahler parameter \(\kappa\), in case (S);
\item that of branch (ii), with K\"ahler parameter \(0\), in case (K).
\end{itemize}
\item[(ii)] Let \(S_0\) be as in \eqref{eq:S0} with \(\sigma_2=0\),
\(\sigma_5=\sigma_3\) and \(\sigma_6=\sigma_4\). Then \(\nabla^g+S\) is
torsion-free and infinitesimally homogeneous, with inner torsion
parameters
\[
(a,b,\alpha,\beta,\gamma,\delta)
=(\sigma_3,\ \sigma_7,\ \beta_T+\sigma_3,\ -\alpha_T+\sigma_4,\
-\beta_T+\sigma_1,\ \alpha_T),
\]
and with the same K\"ahler parameter as \(\nabla^g\).
\end{enumerate}
\end{prop}

\begin{proof}
(i) \emph{Transitivity.}
\begin{itemize}
\item In case (C), the holomorphic isometries of \(N^{2n}(\kappa)\), whose isotropy is
the full unitary group, together with the translations of
\(\mathbb R\), suffice.
\item In case (S), Sasakian space forms have automorphism group of the
maximal dimension \((n+1)^2\)~\cite{Tanno69}. Since the isotropy of a
point is a closed subgroup of \(\operatorname{U}(n)\), of dimension at
most \(n^2\), the orbits are open. The isotropy is therefore a closed
subgroup of \(\operatorname{U}(n)\) of dimension \(n^2\), hence all of
\(\operatorname{U}(n)\). Homotheties and the change of sign preserve
this.
\item In case (K), one uses the translations of \(\mathbb C^n\), the
unitary group, and the maps \((t,z)\mapsto(t+s,e^{-\beta_Ts}z)\).
\end{itemize}

\emph{Inner torsion.} By the first assertion of
Proposition~\ref{prop:local-homogeneity}(c), \((M,\nabla^g,P)\) is
infinitesimally homogeneous. Its inner torsion parameters follow from
Proposition~\ref{prop:levi-civita}, since the structure is
trans-Sasakian of type \((\alpha_T,\beta_T)\).

\emph{Curvature.} As in the proofs of
Theorems~\ref{thm:compatible-torsion-free}
and~\ref{thm:torsion-free-non-compatible}, the compatibility relation
\eqref{eq:torsion-free-compatibility} and the first Bianchi identity
\eqref{eq:first-bianchi} determine the curvature up to a term
\(\kappa'K\) in the horizontal block, for some \(\kappa'\in\mathbb R\).
\begin{itemize}
\item In case (C), \(\kappa'=\kappa\) by
Theorem~\ref{thm:compatible-torsion-free}.
\item In case (S), \eqref{eq:B-branch-i} with \(\Pend=-\alpha_TJ_0\) and
\(\Qend=\alpha_TJ_0\) gives \(B(X,J_0X)J_0X=(\kappa'-3\alpha_T^2)X\) for unit
\(X\). Hence \(c=\kappa'-3\alpha_T^2\), and \(\kappa'=\kappa\).
\item In case (K), the parameter-dependent part of the horizontal block,
with \(\Pend=\beta_T\operatorname{Id}_{H_0}\) and \(\gamma=-\beta_T\),
equals \(-\beta_T^2\bigl(h_0(Y,\cdot)X-h_0(X,\cdot)Y\bigr)\), which is
already the curvature of \(\mathbb H^{2n+1}(-\beta_T^2)\). Hence
\(\kappa'=0\), as also required by the second Bianchi identity, since
\(\alpha-a=\beta_T\neq0\).
\end{itemize}

(ii) By \eqref{eq:S-deformation-TI} and \eqref{eq:S-deformation-R} with
\(T^g=0\), the torsion is \(S_0(u,v)-S_0(v,u)=0\), and
\(\lambda=\lambda^g+S_{0,u}\) gives the stated parameters. In
particular, \(\delta(\alpha-a)=\alpha_T\beta_T=0\), and the new
connection lies in the branch \(a=\alpha\) in cases (C) and (S), and in
the branch \(a\neq\alpha\) in case (K). The coefficients
\(c_0,\dots,c_4\), and the horizontal block up to the term \(\kappa K\),
are determined by the parameters. In both branches this block has the
form
\[
B=B_\lambda-2\beta\delta\,\omega_0(\cdot,\cdot)J_0+\kappa K,
\]
with \(B_\lambda\) as in the branch \(a=\alpha\); in the branch
\(a\neq\alpha\), where \(\delta=0\), this is \eqref{eq:B-branch-ii}
together with \(\kappa=0\). The same holds for \(\nabla^g\) by (i). By
\eqref{eq:S-deformation-R}, the horizontal block of \(\nabla^g+S\) is
\[
B(X,Y)Z=B^g(X,Y)Z+\bigl(D_{\lambda^g(X)}S_0\bigr)(Y,Z)_H
-\bigl(D_{\lambda^g(Y)}S_0\bigr)(X,Z)_H+[S_{0,X},S_{0,Y}]Z .
\]
A direct computation from the explicit forms of \(\lambda^g\) and
\(S_0\) gives, in all three models,
\[
\begin{aligned}
&\bigl(D_{\lambda^g(X)}S_0\bigr)(Y,Z)_H
-\bigl(D_{\lambda^g(Y)}S_0\bigr)(X,Z)_H+[S_{0,X},S_{0,Y}]Z\\
&\qquad=\bigl(B_\lambda-B_{\lambda^g}\bigr)(X,Y)Z
-2\bigl(\beta\delta-\beta^g\delta^g\bigr)\,\omega_0(X,Y)J_0Z,
\end{aligned}
\]
where \(\beta-\beta^g=\sigma_4\) and \(\delta=\delta^g=\alpha_T\).
Adding \(B^g=B_{\lambda^g}-2\beta^g\delta^g\,\omega_0(\cdot,\cdot)J_0
+\kappa K\) gives
\(B=B_\lambda-2\beta\delta\,\omega_0(\cdot,\cdot)J_0+\kappa K\), with
the same \(\kappa\) as \(\nabla^g\).
\end{proof}

\begin{teo}[Structure theorem]
\label{thm:structure}
Let \((M,\nabla,P)\) be infinitesimally homogeneous with structure group
\(\operatorname{U}(n)\times\{1\}\) and characteristic parameters
\(\tau_1,\tau_2,\tau_3,a,b,\alpha,\beta,\gamma,\delta\), and put
\[
\alpha_T:=\delta-\frac{\tau_3}{2},
\qquad
\beta_T:=\alpha-a-\tau_1 .
\]
\begin{enumerate}
\item[(a)] The Levi-Civita connection \(\nabla^g\) of
\(g=h+\eta\otimes\eta\) is infinitesimally homogeneous with respect to
\(P\). The structure \((\phi,\xi,\eta,g)\) is trans-Sasakian of type
\((\alpha_T,\beta_T)\), and \(\alpha_T\beta_T=0\).
\item[(b)] \((M,P)\) is locally isomorphic to exactly one of the models
(C), (S), (K). Here \(\kappa\) is the K\"ahler parameter of
Corollary~\ref{cor:parameter-classification}.
\item[(c)] \(\nabla=\nabla^g+S\), where \(S\) is induced by a unique
\(S_0\) as in \eqref{eq:S0}, and
\begin{equation}\label{eq:sigma-dictionary}
\begin{gathered}
\tau_1=\sigma_3-\sigma_5,\quad
\tau_2=\sigma_4-\sigma_6,\quad
\tau_3=2\sigma_2,\\
a=\sigma_5,\quad b=\sigma_7,\quad
\alpha=\beta_T+\sigma_3,\quad
\beta=-\alpha_T+\sigma_4,\quad
\gamma=-\beta_T+\sigma_1,\quad
\delta=\alpha_T+\sigma_2 .
\end{gathered}
\end{equation}
\item[(d)] Conversely, for every model and every \(S_0\) as in
\eqref{eq:S0}, the connection \(\nabla^g+S\) is infinitesimally
homogeneous with respect to \(P\).
\end{enumerate}
\end{teo}

\begin{proof}
\emph{(a) and (c).} By Proposition~\ref{prop:local-homogeneity}(c),
\(\nabla^g\) is infinitesimally homogeneous, and \(S:=\nabla-\nabla^g\)
is induced by a \(G\)-equivariant \(S_0\), which has the form
\eqref{eq:S0}. By Proposition~\ref{prop:levi-civita}, the structure is
trans-Sasakian of some type \((\alpha_T,\beta_T)\) with
\(\alpha_T\beta_T=0\), and \(\nabla^g\) has parameters
\((0,0,\beta_T,-\alpha_T,-\beta_T,\alpha_T)\).
Formula \eqref{eq:S-deformation-TI}, applied with base \(\nabla^g\),
gives \eqref{eq:sigma-dictionary}. Eliminating the
\(\sigma_i\) yields \(\alpha_T=\delta-\tau_3/2\) and
\(\beta_T=\alpha-a-\tau_1\).

\emph{(b).} We first classify \(\nabla^g\) by itself. By (a),
\((M,\nabla^g,P)\) is torsion-free and infinitesimally homogeneous, with
inner torsion parameters \((0,0,\beta_T,-\alpha_T,-\beta_T,\alpha_T)\).
By Theorems~\ref{thm:compatible-torsion-free}
and~\ref{thm:torsion-free-non-compatible}, its curvature is determined by
these parameters and by a K\"ahler parameter \(\kappa_g\), with
\(\kappa_g=0\) if \(\beta_T\neq0\). By Proposition~\ref{prop:models}(i),
the model with parameters \((\alpha_T,\beta_T,\kappa_g)\), endowed with
its Levi-Civita connection, has the same characteristic tensors. Hence
Proposition~\ref{prop:local-homogeneity}(d) gives a local isomorphism
\(f\) of \(G\)-structures from \(M\) to this model, which maps
\(\nabla^g\) to the Levi-Civita connection of the model. The three cases
are distinguished by \((\alpha_T,\beta_T)\).

It remains to compare \(\kappa_g\) with the K\"ahler parameter \(\kappa\)
of Corollary~\ref{cor:parameter-classification}, that is, with the
K\"ahler parameter of the symmetrized connection
\(\nabla^s=\nabla-\frac12T=\nabla^g+S^{\mathrm{sym}}\), where
\(S^{\mathrm{sym}}(X,Y)=\frac12\bigl(S(X,Y)+S(Y,X)\bigr)\). The tensor
\(S^{\mathrm{sym}}\) is induced by a symmetric element of
\eqref{eq:S0}, that is, one with \(\sigma_2=0\), \(\sigma_5=\sigma_3\)
and \(\sigma_6=\sigma_4\). Transporting by \(f\), \(\nabla^s\)
corresponds to the connection \(\nabla^g+S^{\mathrm{sym}}\) on the
model, and Proposition~\ref{prop:models}(ii) gives \(\kappa=\kappa_g\).

\emph{(d).} This follows from Proposition~\ref{prop:models}(i) and the
deformation formulas \eqref{eq:S-deformation-TI} and
\eqref{eq:S-deformation-R}.
\end{proof}

\begin{cor}\label{cor:invariants}
In the situation of Theorem~\ref{thm:structure}:
\begin{enumerate}
\item[(1)] the relation \eqref{eq:parameter-relation} of
Corollary~\ref{cor:parameter-classification} is equivalent to
Marrero's identity \(\alpha_T\beta_T=0\);
\item[(2)] for \(X,Y\in H\),
\[
d\eta(X,Y)=(\tau_3-2\delta)\,\omega(X,Y)=-2\alpha_T\,\omega(X,Y).
\]
In particular, \(H\) is a contact distribution, that is,
\(\eta\wedge(d\eta)^n\neq0\), if and only if \(\tau_3\neq2\delta\), that
is, if and only if the model is (S); in that case \(\alpha=a+\tau_1\);
\item[(3)] the K\"ahler parameter \(\kappa\) depends only on the
underlying \(G\)-structure, not on the connection: it is a local
invariant of \((M,P)\). It is arbitrary in cases (C) and (S), where
it equals \(c+3\alpha_T^2\) in terms of the \(\phi\)-sectional
curvature \(c\), and it vanishes in case (K).
\end{enumerate}
\end{cor}

\begin{proof}
(1) By definition, \(2\delta-\tau_3=2\alpha_T\) and
\(\alpha-a-\tau_1=\beta_T\).

(2) Since \(\nabla^g\) is torsion-free and metric,
\(d\eta(X,Y)=g(\nabla^g_X\xi,Y)-g(\nabla^g_Y\xi,X)\). For \(X,Y\in H\),
the trans-Sasakian formula
\(\nabla^g_X\xi=-\alpha_T\phi X+\beta_TX\) gives
\(d\eta(X,Y)=-2\alpha_T\,\omega(X,Y)\), and
\(-2\alpha_T=\tau_3-2\delta\). Since \(d\eta+2\alpha_T\,\omega\)
vanishes on \(H\times H\) and \(\iota_\xi\omega=0\), there is a
\(1\)-form \(\mu\) with \(d\eta=-2\alpha_T\,\omega+\eta\wedge\mu\).
Hence \(\eta\wedge(d\eta)^n=(-2\alpha_T)^n\,\eta\wedge\omega^n\), and
\(\eta\wedge\omega^n\) is a volume form. If \(\alpha_T\neq0\), then
\(\beta_T=0\) by (1), that is, \(\alpha=a+\tau_1\).

(3) By Theorem~\ref{thm:structure}(b), \(\kappa\) is the K\"ahler
parameter of \(\nabla^g\), which depends only on \(P\). The remaining
assertions follow from Proposition~\ref{prop:models}(i).
\end{proof}

By \eqref{eq:inner-parameter-transformation}, \(\alpha_T=\delta^s\) and
\(\beta_T=\alpha^s-a^s\): the model underlying \((M,\nabla,P)\) is read
off from its symmetrized connection, as it must be, since \(\nabla\)
and \(\nabla^s\) have the same \(G\)-structure. In particular, \(H\) is
contact precisely when the symmetrized connection lies in the branch
\(a^s=\alpha^s\) with \(\delta^s\neq0\).

Marrero~\cite{marrero} proved that every trans-Sasakian manifold of
dimension at least five is cosymplectic, \(\alpha_T\)-Sasakian or
\(\beta_T\)-Kenmotsu, while in dimension three one only has
\(\xi(\alpha_T)=-2\alpha_T\beta_T\), so that constant coefficients
again force \(\alpha_T\beta_T=0\). In the present setting the relation
follows uniformly for all \(n\geq1\) from the algebraic condition
\eqref{eq:delta-alpha-a}.

\begin{remark}\label{rem:CG-Tanno}
In the decomposition of Chinea and Gonz\'alez~\cite{ChineaGonzalez} of
the intrinsic torsion of almost contact metric structures into twelve
classes \(\mathcal C_1,\dots,\mathcal C_{12}\), the trans-Sasakian
structures are precisely those of class
\(\mathcal C_5\oplus\mathcal C_6\)~\cite{oubina}, the classes
\(\mathcal C_6\) and \(\mathcal C_5\) corresponding to the
\(\alpha\)-Sasakian and \(\beta\)-Kenmotsu structures. By
Theorem~\ref{thm:structure}(a), all the structures in our classification
belong to this class.

Tanno's classification~\cite{Tanno69} of almost contact metric manifolds
whose automorphism group has the maximal dimension \((n+1)^2\) is
global. Up to local isomorphism and homothety, it reduces to the three
families (C), (S), (K), and Theorem~\ref{thm:structure} may thus be read
as a local, connection-level refinement of Tanno's classification.
\end{remark}

\begin{center}
\small
\renewcommand{\arraystretch}{1.3}
\begin{tabular}{c|c|c|c}
model & \((\alpha_T,\beta_T)\) & underlying manifold & parameters\\ \hline
(C) & \((0,0)\) & \(N^{2n}(\kappa)\times\mathbb R\) &
\(\kappa;\ \sigma_1,\dots,\sigma_7\)\\
(S) & \((\alpha_T\neq0,\,0)\) & \(\alpha_T\)-Sasakian space form & \(\alpha_T,\kappa;\ \sigma_1,\dots,\sigma_7\)\\
(K) & \((0,\,\beta_T\neq0)\) & \(\mathbb H^{2n+1}(-\beta_T^2)\), Kenmotsu &
\(\beta_T;\ \sigma_1,\dots,\sigma_7\)
\end{tabular}
\end{center}

\begin{cor}[Compatible connections]
\label{cor:compatible-connections}
Let \((M,\nabla,P)\) be infinitesimally homogeneous with structure group
\(\operatorname{U}(n)\times\{1\}\) and \(\mathfrak I_0=0\). Then
\[
\tau_1\tau_3=0,
\]
and the K\"ahler parameter \(\kappa\in\mathbb R\) of the symmetrized
connection satisfies
\[
\kappa\tau_1=0.
\]
Moreover,
\[
R_0(X,e)=0,
\quad
R_0(X,Y)e=0,
\]
and
\[
R_0(X,Y)Z
=
\kappa K_{X,Y}Z
-\tau_2\tau_3\,\omega_0(X,Y)J_0Z.
\]
Conversely, every
\((\tau_1,\tau_2,\tau_3,\kappa)\in\mathbb R^4\)
satisfying
\[
\tau_1\tau_3=0,
\quad
\kappa\tau_1=0
\]
occurs.

In particular, if \(\tau_1\neq0\), then
\(\tau_3=\kappa=0\) and \(\nabla\) is flat. The underlying model is (S)
with \(\alpha_T=-\tau_3/2\) if \(\tau_3\neq0\), (K) with
\(\beta_T=-\tau_1\) if \(\tau_1\neq0\), and (C) otherwise.
\end{cor}

\begin{proof}
Write \(\kappa=\kappa^s\). Since \(\mathfrak I_0=0\), one has
\(a=b=\alpha=\beta=\gamma=\delta=0\) and \(\Pend=\Qend=0\). Relation
\eqref{eq:parameter-relation} reduces to \(\tau_1\tau_3=0\), and the
K\"ahler parameter is unrestricted only when \(\alpha-a=\tau_1\), that
is, when \(\tau_1=0\); thus \(\kappa\tau_1=0\). Since \(\nabla\)
preserves \(P\), its curvature takes values in
\(\mathfrak g=\mathfrak u(n)\oplus\{0\}\), so \(R_0(X,Y)e=0\); by
Lemma~\ref{lem:parity}, \(R_0(X,e)\in\mathfrak{gl}(V)_-\), whereas
\(\mathfrak g\subset\mathfrak{gl}(V)_+\), hence \(R_0(X,e)=0\).

By \eqref{eq:inner-parameter-transformation}, the symmetrized
connection has \(a^s=-\alpha^s=\tau_1/2\), \(b^s=\gamma^s=0\),
\(\beta^s=-\tau_2/2\) and \(\delta^s=-\tau_3/2\). If \(\tau_1=0\) and
\((\tau_2,\tau_3)\neq(0,0)\), it lies in the non-compatible branch
\(a^s=\alpha^s\), and \eqref{eq:B-deformed-branch-i} gives
\(B(X,Y)=\kappa K_{X,Y}-\tau_2\tau_3\,\omega_0(X,Y)J_0\); if
\(\tau_1=\tau_2=\tau_3=0\), it is compatible, and
Corollary~\ref{cor:deformation-compatible-family} gives the same
formula; if \(\tau_1\neq0\), then \(\tau_3=\kappa=0\), it lies in the
branch \(a^s\neq\alpha^s\), and \eqref{eq:B-deformed-branch-ii} gives
\(B=0\), in agreement with the formula.

Conversely, by Corollary~\ref{cor:parameter-classification}, every
\((\tau_1,\tau_2,\tau_3,\kappa)\) satisfying \(\tau_1\tau_3=0\) and
\(\kappa\tau_1=0\) occurs. The last assertion follows from
Theorem~\ref{thm:structure}, since \(\alpha_T=\delta-\tau_3/2=-\tau_3/2\)
and \(\beta_T=\alpha-a-\tau_1=-\tau_1\).
\end{proof}

\begin{remark}[Compatible connections on Sasakian space forms]
\label{rem:classical-contact-models}
A contact metric structure in the sense of~\cite{Blair} satisfies
\(\Phi=d\eta\), with the exterior derivative normalized as
in~\cite{Blair}, which differs from ours by a factor \(\frac12\); with
our conventions this reads \(d\eta=2\Phi=-2\omega\). By
Corollary~\ref{cor:invariants}(2), this means \(\alpha_T=1\), so
\(\tau_3=-2\) for compatible connections. On a Sasakian space form of
constant \(\phi\)-sectional curvature \(c\), the Tanaka--Webster
connection \(\nabla^*\)~\cite{tanno} preserves \(\phi,\xi,\eta,g\), and
\(T^*(X,Y)=2\Phi(X,Y)\xi=-2\omega(X,Y)\xi\) for \(X,Y\in H\),
\(T^*(\xi,X)=0\). Being canonically attached to the structure,
\(\nabla^*\) is invariant under its automorphisms, which act
transitively on the frames of \(P\) by
Proposition~\ref{prop:models}(i). Hence it is infinitesimally
homogeneous, with \(\mathfrak I_0=0\),
\((\tau_1,\tau_2,\tau_3)=(0,0,-2)\) and \(R_0=\kappa K\), where
\(\kappa=c+3\) by Corollary~\ref{cor:invariants}(3). Thus the
Heisenberg group (\(c=-3\)), the round sphere \(\mathbb S^{2n+1}\)
(\(c=1\)) and the Sasakian space forms with \(c<-3\) correspond to
\(\kappa=0\), \(\kappa=4\) and \(\kappa<0\); up to homothety, their
contact sub-Riemannian structures are the symmetric spaces
of~\cite{falbel-gorodski} fibred over \(\mathbb C^n\),
\(\mathbb{CP}^n\) and \(\mathbb{CH}^n\). The characteristic connection
of~\cite{friedrich-ivanov}, whose torsion \(3\)-form is
\(\eta\wedge d\eta\), corresponds to \((\tau_2,\tau_3)=(2,-2)\); in
general, the torsion is totally skew-symmetric precisely when
\(\tau_2=-\tau_3\).

In Corollary~\ref{cor:compatible-connections}, the term
\(-\tau_2\tau_3\,\omega_0J_0\) is the one forced by the first Bianchi
identity \(\mathfrak S\,R_0(X,Y)Z=\mathfrak S\,T_0(T_0(X,Y),Z)\), valid
because \(\nabla T=0\). For \(n=1\), one has
\(\omega_0(X,Y)J_0=-K_{X,Y}\), so that
\(R_0(X,Y)=(\kappa+\tau_2\tau_3)K_{X,Y}\) on \(H_0\); thus the
horizontal curvature is \(\kappa+\tau_2\tau_3\) rather than \(\kappa\).
\end{remark}

\begin{remark}
\label{rem:non-contact-compatible-deformations}
The torsion alone does not determine whether \(H\) is contact: by
Corollary~\ref{cor:invariants}(2), this is governed by
\(\tau_3-2\delta\), not by \(\tau_3\). For instance, the deformations of
the compatible torsion-free family
(Corollary~\ref{cor:deformation-compatible-family}) have
\(\delta^s=0\) and \(\alpha^s=a^s\), hence lie over the coK\"ahler
model (C) for every \(\tau_3\). Taking \(\tau_1=\tau_2=0\),
\(\tau_3=-2\) and \(\kappa^s=0\) gives a flat connection with
\(T_0(X,Y)=-2\,\omega_0(X,Y)e\) and \(T_0(X,e)=0\), which are exactly
the torsion and curvature of the Tanaka--Webster connection of the
Heisenberg group. The two triples are nevertheless not locally
equivalent, since here \(d\eta|_{H\times H}=0\), whereas the Heisenberg
structure is contact: the same pair \((T_0,R_0)\) occurs over different
\(G\)-structures, and the difference is detected only by the inner
torsion, \(\delta=-1\) versus \(\delta=0\).
\end{remark}

\subsection{Ambient spaces for affine immersions}
\label{subsec:immersions}

By Theorem~\ref{thm:structure}, the possible ambient spaces with
structure group \(\operatorname{U}(n)\times\{1\}\) for the immersion
problem of Piccione and Tausk~\cite{piccione-tausk} are, locally, the
three models (C), (S), (K), each endowed with a connection
\(\nabla^g+S\). The Gauss, Codazzi, Ricci and torsion equations
of~\cite{piccione-tausk} then have coefficients that are polynomial in
\((\alpha_T,\beta_T,\kappa)\) and \(\sigma_1,\dots,\sigma_7\). For
instance, for compatible connections on Sasakian space forms they
depend only on \((\tau_2,\kappa)\). A systematic study of these
equations, which would extend the classical existence theorems for
isometric immersions into Sasakian space forms, is left for future
work.

\section*{Declarations}

\noindent\textbf{Use of AI tools.}
During the preparation of this work the authors used the AI assistant
Claude (Anthropic) for the following purposes: (i) to verify
computations symbolically and numerically, in particular the curvature,
Bianchi-identity and deformation formulas of
Section~\ref{sec:unitary-one}; (ii) to suggest arguments and the
organization of some parts, namely the identification of
relation~\eqref{eq:parameter-relation} with Marrero's identity and the
resulting structure theorem (Subsection~\ref{subsec:underlying},
Theorem~\ref{thm:structure} and Corollary~\ref{cor:invariants}), the
realization of the torsion-free data through left-invariant connections
on Lie groups (Lemma~\ref{lem:lie-realization}), the organization of
Subsection~\ref{subsec:unitary-torsion-deformations} into deformation
corollaries, and the examples in
Remarks~\ref{rem:classical-contact-models}
and~\ref{rem:non-contact-compatible-deformations}; and (iii) to improve
the language and the \LaTeX{} presentation throughout. All mathematical
statements and proofs were checked by the authors, who take full
responsibility for the content of the paper.

\medskip
\noindent\textbf{Funding.} This research received no specific funding.

\medskip
\noindent\textbf{Competing interests.} The authors have no competing
interests to declare.

\medskip
\noindent\textbf{Data availability.} Not applicable; no data were
generated or analysed in this study.

%%%%%%%%%%%%%%%%%%%%%%%%
%%%%%%%%%%%%%%%%%%%%%%%%%


\begin{thebibliography}{99}

\bibitem{AmbroseSinger}
W. Ambrose and I.~M. Singer, On homogeneous Riemannian manifolds,
Duke Math. J. {\bf 25} (1958), 647--669.

\bibitem{Blair}
D.~E. Blair, {\it Riemannian geometry of contact and symplectic
manifolds}, second edition, Progress in Mathematics, 203,
Birkh\"auser Boston, Boston, MA, 2010.

\bibitem{ChineaGonzalez}
D. Chinea and J.~C. Gonz\'alez-D\'avila, A classification of almost
contact metric manifolds, Ann. Mat. Pura Appl. (4) {\bf 156} (1990),
15--36.

\bibitem{falbel-gorodski}
E. Falbel and C. Gorodski, On contact sub-Riemannian symmetric spaces,
Ann. Sci. \'Ecole Norm. Sup. (4) {\bf 28} (1995), no.~5, 571--589.

\bibitem{friedrich-ivanov}
T. Friedrich and S.~P. Ivanov, Parallel spinors and connections with
skew-symmetric torsion in string theory, Asian J. Math. {\bf 6}
(2002), no.~2, 303--335.

\bibitem{Kenmotsu}
K. Kenmotsu, A class of almost contact Riemannian manifolds, Tohoku
Math. J. (2) {\bf 24} (1972), 93--103.

\bibitem{KobayashiNomizuI}
S. Kobayashi and K. Nomizu, {\it Foundations of differential geometry.
Vol I}, Interscience Publishers (a division of John Wiley \& Sons,
Inc.), New York-London, 1963.

\bibitem{KobayashiNomizuII}
S. Kobayashi and K. Nomizu, {\it Foundations of differential geometry.
Vol. II}, Interscience Tracts in Pure and Applied Mathematics, No.~15,
Interscience Publishers John Wiley \& Sons, Inc., New
York-London-Sydney, 1969.

\bibitem{marin}
C.~A. Mar\'in~Arango, An algebraic characterization of affine
manifolds with $G$-structure satisfying a homogeneity condition, Rev.
Colombiana Mat. {\bf 44} (2010), no.~2, 149--165.

\bibitem{Marin-Blazquez}
C.~A. Mar\'in~Arango and D. Bl\'azquez-Sanz, Infinitesinally
homogeneous manifolds with prescribed structure groups, Rev. Colombiana
Mat. {\bf 50} (2016), no.~1, 1--15.

\bibitem{marrero}
J.~C. Marrero~Gonz\'alez, The local structure of trans-Sasakian
manifolds, Ann. Mat. Pura Appl. (4) {\bf 162} (1992), 77--86.

\bibitem{Opozda97}
B. Opozda, Affine versions of Singer's theorem on locally homogeneous
spaces, Ann. Global Anal. Geom. {\bf 15} (1997), no.~2, 187--199.

\bibitem{Opozda98}
B. Opozda, On locally homogeneous $G$-structures, Geom. Dedicata
{\bf 73} (1998), no.~2, 215--223.

\bibitem{oubina}
J.~A. Oubi\~na, New classes of almost contact metric structures, Publ.
Math. Debrecen {\bf 32} (1985), no.~3-4, 187--193.

\bibitem{piccione-tausk}
P. Piccione and D.~V. Tausk, An existence theorem for $G$-structure
preserving affine immersions, Indiana Univ. Math. J. {\bf 57} (2008),
no.~3, 1431--1465.

\bibitem{Singer}
I.~M. Singer, Infinitesimally homogeneous spaces, Comm. Pure Appl.
Math. {\bf 13} (1960), 685--697.

\bibitem{Tanno69}
S. Tanno, The automorphism groups of almost contact Riemannian
manifolds, Tohoku Math. J. (2) {\bf 21} (1969), 21--38.

\bibitem{tanno}
S. Tanno, Variational problems on contact Riemannian manifolds, Trans.
Amer. Math. Soc. {\bf 314} (1989), no.~1, 349--379.

\bibitem{TricerriVanhecke}
F. Tricerri and L. Vanhecke, {\it Homogeneous structures on Riemannian
manifolds}, London Mathematical Society Lecture Note Series, 83,
Cambridge Univ. Press, Cambridge, 1983.

\end{thebibliography}
\end{document}